\documentclass[12pt]{article}
 \usepackage[english]{babel}
  \usepackage{amsmath,amsfonts,amssymb,amsthm,bbm}
\usepackage{graphicx}

\usepackage[latin1]{inputenc}
\usepackage[OT1]{fontenc}

\usepackage[active]{srcltx} %%allows you to jump from your editor directly into
  \usepackage[colorlinks=false]{hyperref}
\usepackage{paralist,subfigure}

\usepackage{amsfonts}
\usepackage{amsmath}
 \usepackage{psfrag}
  \usepackage{epsfig}

\def\R{\mathbb R}
\def\N{\mathbb N}

\def\Q{\mathbb Q}
\def\C{\mathbb C}

\def\1{{\mathbbm 1}}
\def\supp{{\mathrm {supp}}}
\def\dist{{\mathrm {dist}}}

\def\conv{{\mathrm {conv}}}
\renewcommand\Re{{\mathrm {Re}}\,}
\renewcommand\Im{{\mathrm {Im}}\,}
\def\Arg{{\mathrm {Arg}}\,}

\let\O=\Omega
\let\e=\varepsilon
\let\vp=\varphi
\let\vt=\vartheta
\let\t=\tilde
\let\ol=\overline
\let\ul=\underline
\let\.=\cdot
\let\0=\emptyset

\let\mc=\mathcal
\let\di=\displaystyle

\def\sse{\ \Leftrightarrow\ }

\def\pe{principal eigenvalue}

\newcommand{\su}[2]{\genfrac{}{}{0pt}{}{#1}{#2}}
\newcommand{\rest}[2]{\!\left.\!\hspace{1pt}#1\right|_{#2}}

\def\thm#1{Theorem \ref{thm:#1}}
\def\seq#1{(#1_n)_{n\in\N}}
\def\limn{\lim_{n\to\infty}}

\newenvironment{formula}[1]{\begin{equation}\label{#1}}
                       {\end{equation}\noindent}

\def\Fi#1{\begin{formula}{#1}}
\def\Ff{\end{formula}\noindent}

\newtheorem{theorem}{Theorem}[section]
\newtheorem{proposition}[theorem]{Proposition}
\newtheorem{lemma}[theorem]{Lemma}

\theoremstyle{definition}

\newtheorem{definition}[theorem]{Definition}

\title{\bf The shape of expansion induced by a line with fast diffusion in Fisher-KPP equations}
\author{Henri {\sc Berestycki}$^{\hbox{a }}$,  
Jean-Michel {\sc Roquejoffre}$^{\hbox{b }}$, Luca {\sc Rossi}$^{\hbox{c }}$\\
\footnotesize{$^{\hbox{a }}$ Ecole des Hautes Etudes en Sciences Sociales}\\
\footnotesize{ CAMS, 54, bd Raspail F-75270 Paris, France}\\
\footnotesize{$^{\hbox{b }}$ Institut de Math\'ematiques de Toulouse,
Universit\'e Paul Sabatier}\\
\footnotesize{118 route de Narbonne, F-31062 Toulouse Cedex 4, France}\\
\footnotesize{$^{\hbox{c }}$Universit\`a degli Studi di Padova}\\
\footnotesize{Dipartimento di Matematica, Via Trieste, 63 -
35121 Padova, Italy}\\
}

\begin{document}

\maketitle

\def\W{\mc{W}}

%%%%%%%%%%%%%%%%%%%%%%%%%%%%%%%%%%%%%%%%%%%%%%%%%%%%%%%%%%%%%%%%%%

\begin{abstract}
We establish a new property of Fisher-KPP type propagation in a
plane, in the presence of a line with fast diffusion. We prove that the line 
enhances the asymptotic speed of propagation
in a cone of directions. Past the critical angle given by this cone, the 
asymptotic speed of
propagation coincides with the classical Fisher-KPP invasion speed. Several qualitative
properties  are further derived, such as the limiting behaviour when the diffusion on the line goes
to infinity.
\end{abstract}

\section{Introduction}

In \cite{BRR2} we introduced a new model to
describe biological invasions in the
plane when a strong diffusion takes place on a straight line.  
In this model, we consider a coordinate system on $\R^2$ with the $x$-axis
coinciding with the line, referred to as ``the road''. The rest of the plane is
called ``the field''.
For given time $t\geq0$, we let $v(x,y,t)$ denote the density
of the population at the point $(x,y)\in\R^2$ of the field
and $u(x,t)$ denote the density at the point $x\in\R$ of the road. 
Owing to the symmetry of the problem, one can restrict the field to the upper
half-plane $\O:=\R\times(0,+\infty)$.
There, the dynamics is assumed to be given by a standard Fisher-KPP equation 
with diffusivity~$d$, whereas, on the road, there is no reproduction nor 
mortality
and the diffusivity is given by another constant $D$. We are especially
interested in the case
where $D$ is much larger than $d$.
On the vicinity of the road there is a constant exchange between the densities $u$, and the one in the field adjacent
to the road, $\rest{v}{y=0}$, given by two rates $\mu,\nu$ respectively. That is, a proportion $\mu$ of $u$ jumps off the road into the field while a proportion $\nu$ of $\rest{v}{y=0}$ goes onto the road.

This model gives rise to the following system:
\Fi{Cauchy}
\begin{cases}
\partial_t u-D \partial_{xx} u= \nu\rest{v}{y=0}-\mu u, &
x\in\R,\
t>0\\
\partial_t v-d\Delta v=f(v), & (x,y)\in\O,\ t>0\\
-d\partial_y\rest{v}{y=0}=\mu u-  \nu\rest{v}{y=0}, & x\in\R,\ t>0,
\end{cases}
\Ff
where $d,D,\mu,\nu$ are positive constants and $f\in C^1([0,+\infty))$
satisfies the usual KPP type assumptions:
$$f(0)=f(1)=0,\quad f>0\text{ in }(0,1),\quad f<0\text{ in
}(1,+\infty),\quad f(s)\leq f'(0)s \text{ for }s>0.$$
These hypotheses will always be understood in the following without further
mention.
We complete the system with initial conditions:
$$
\rest{u}{t=0}=u_0\quad\text{in }\R,\qquad
\rest{v}{t=0}=v_0\quad\text{in }\O,
$$
where $u_0$, $v_0$ are always assumed to be nonnegative, bounded and continuous.
The existence of a classical solution for this Cauchy problem has been derived 
in
\cite{BRR2}, together with the weak and strong comparison principles. 

Let $c_K$ denote the KPP spreading velocity (or invasion speed) \cite{KPP} in
the field:
$$
c_K=2\sqrt{df'(0)}.
$$
This is the asymptotic speed at which the population would spread in any direction in the open
space - i.e., when the road is not present (see \cite{AW0}, \cite{AW}).
\\
The question that we treat in this paper is the following. In \cite{BRR2} (c.f.~also Theorem 1.1 in
\cite{BRR3}) we proved that there exists
$c_*\geq c_K$ such that, if $(u,v)$ is the solution of \eqref{Cauchy} 
emerging from $(u_0,v_0)\not\equiv(0,0)$, there holds that
\begin{equation}
\label{e1.5}
\begin{array}{rll}
&\displaystyle \forall c>c_*,\quad\
\lim_{t\to+\infty}\sup_{\su{|x|>ct}{y\geq0}}
|(u(x,t),v(x,y,t))|=0,\\
&\displaystyle \forall c<c_*,\ a>0,\quad\ 
\lim_{t\to+\infty}\sup_{\su{|x|<ct}{0\leq 
y<a}}|(u(x,t),v(x,y,t))-(\nu/\mu,1)|=0.
\end{array}
\end{equation}
Moreover, $c_*>c_K$ if and only if $D>2d$. In other words, the solution spreads 
at velocity $c_*$ in the direction of the road.

Clearly, the convergence of $v$ to $1$ in the second limit cannot hold uniformly
in~$y$. The purpose of this
paper is precisely to understand the asymptotic limits in various directions, 
and this turns out to be a rather  delicate issue.
Here is one of our main results.
\begin{theorem}
\label{t1.1}
There exists $w_*\in C^1([-\pi/2,\pi/2])$ such that
$$\forall c>w_*(\vt),\quad
\lim_{t\to+\infty}v(x_0+ ct\sin\vt, y_0+ ct\cos\vt,t)=0,$$
$$\forall0\leq c<w_*(\vt),\quad\lim_{t\to+\infty}v(x_0+ct\sin\vt,y_0+ 
ct\cos\vt,t)=1,$$
locally uniformly in $(x_0, y_0)\in\ol\O$ and uniformly in 
$(c,\vt)\in\R_+\times[-\pi/2,\pi/2]$ such that
$|c-w_*(\vt)|>\e$, for any given $\e>0$.

Moreover, $w_*\geq c_K$ and, if $D> 2d$,  there is $\vt_0\in(0,\pi/2)$ such 
that $w_*(\vt)>c_K$ if and only if $|\vt|>\vt_0$.
\end{theorem}
\noindent In other words, this theorem provides the spreading velocity in every 
direction $(\sin \vt,\cos\vt)$, and reveals a critical angle phenomenon: the 
road influences 
the propagation on the field much further than just in the horizontal direction.
In Section~\ref{sec:main}, we state a slightly more general result, 
\thm{main}.
% 
% We establish here the existence of an asymptotic speed of spreading in 
% any given direction $\xi$. Namely, the existence of a quantity $w^*$ such that
% $$\forall c>w_*,\quad
% \lim_{t\to+\infty}v(z_0+ ct\,\xi,t)=0,$$
% $$\forall0\leq c<w_*,\quad\lim_{t\to+\infty}v(x_0+ct\sin\vt,y_0+ 
% ct\cos\vt,t)=1,$$
% locally uniformly in $z_0$. The asymptotic speed $w^*$ depends on the direction 
% $\xi$, and the above limits actually hold uniformly in $\xi$ and in $c$ bounded 
% away from $w^*$.
% We further show that in the case where the road enhances the invasion speed in 
% the horizontal direction, i.e.~when $D>2d$, it also enhances the speed in a 
% cone of directions around the horizontal one. Outside this cone the road has no 
% effect on the asymptotic speed, which coincides with the standard one $c_K$.
%  below.
% They reveal a critical angle phenomenon and show that the road influences the 
% propagation on the field much further than just in the horizontal direction.

The paper is organised as follows. In Section \ref{sec:main} we state the main 
results and discuss them. In Section \ref{sec:exp} 
we compute the planar waves of system \eqref{Cauchy} linearised around $v\equiv 
0$. In 
Section \ref{sec:subsolutions}, we construct compactly supported subsolutions to 
\eqref{Cauchy}, based on the already computed planar waves. This is perhaps 
the most technical part of the paper, but which yields a lot of of information 
about the system. The main result, that is, the asymptotic spreading velocity 
in every direction, is proved in Section \ref{sec:spreading}. Section 
\ref{sec:w*} 
is devoted to further properties of the asymptotic speed in therms of the angle 
of the spreading directions with the road. Finally, Section \ref{sec:T} 
describes the modifications that should be made when further effects, namely 
transport and mortality on the road, are included. A comparison result between 
generalised sub and supersolutions is given in the appendix.

% This, roughly speaking, means that solutions $(u,v)$ emerging from
% nontrivial, compactly supported initial data converge as $t\to+\infty$ to the
% unique positive stationary solution
% $(\nu/\mu,1)$. Moreover, there is an asymptotic speed in the direction of
% the line at which this convergence takes place. This, roughly speaking, means
% that on points moving with speed $c$ in the direction $(\pm 1,0)$, $(u,v)$
% converges to $(0,0)$ if $c$ is larger than this asymptotic speed, whereas it
% converges to $(\nu/\mu,1)$ if $c$ is smaller.
% The precise statement will be reclaimed in the next section. Finally, 

%%%%%%%%%%%%%%%%%%%%%%%%%%%%%%%%%%%%%%%%%%%%%%%%%%%%%%%%%%%%%%%%%%

% Let us start by defining this notion precisely.
\section{Statement of results and discussion}\label{sec:main}
\subsection{The main result and some extensions}
We say that \eqref{Cauchy} admits the {\em asymptotic expansion shape} $\W$ if 
any solution $(u,v)$ emerging from a compactly
supported initial datum $(u_0,v_0)\not\equiv(0,0)$ satisfies
% $$\lim_{t\to+\infty}\sup_{\su{x<(w_*^--\e)t}{y\geq0}}
% |(u(x,t),v(x,y,t))|=0,$$
% $$\lim_{t\to+\infty}\sup_{\su{x>(w_*^++\e)t}{y\geq0}}
% |(u(x,t),v(x,y,t))|=0,$$
% $$\forall a>0,\quad\
% \lim_{t\to+\infty}\sup_{\su{(w_*^-+\e)t<x<(w_*^+-\e)t}{
% 0\leq y<a}}|(u(x,t),v(x,y,t))-(U,V(y))|=0,$$
\Fi{W+}
\forall\e>0,\qquad
\lim_{t\to+\infty}\sup_{\su{(x,y)\in\ol\O}{\dist(\frac1t(x,y),\W)>\e}}
v(x,y,t)=0,
\Ff
\Fi{W-}
\forall\e>0,\qquad\lim_{t\to+\infty}\sup_{\su{(x,y)\in\ol\O}{\dist(\frac1t(x,y),
\ol\O\backslash\W)>\e}}|v(x,y,t)-1|=0.
\Ff
Roughly speaking, this means that the upper level sets of $v$
look approximately like $t\W$ for $t$ large enough. Let us emphasise that the
shape $\W$ does not depend on the particular initial datum -- which is a strong property.
In order  for conditions \eqref{W+}, \eqref{W-} in this definition to genuinely make sense (and not be vacuously
satisfied -- think of the set $\W=\Q^2\cap\Omega$),
we further require that the asymptotic expansion shape coincides with the closure of its
interior.
This condition automatically implies that the asymptotic expansion shape is
unique when it exists.
% 
% Further requiring that the
% asymptotic expansion shape is closed makes it unique by definition, when it
% exists. 

In the sequel, we will sometimes consider the polar coordinate system
with the angle taken with respect to the vertical axis. 
Namely, we will write points in the form $r(\sin\vt,\cos\vt)$.
We now state the main result of this paper.
\begin{theorem}\label{thm:main}
Assume the above conditions on $f$.
\begin{enumerate}[(i)]
\item {\em(Spreading).} 
Problem \eqref{Cauchy} admits an asymptotic expansion shape $\W$. 
\item {\em(Shape of $\mc{W}$).} 
The set $\mc{W}$ is convex and it is of the form
$$\mc{W}=\{r(\sin\vt,\cos\vt)\ :\
-\pi/2\leq\vt\leq\pi/2,\ \ 0\leq r\leq w_*(\vt)\}.$$
Here, $w_*\in C^1([-\pi/2,\pi/2])$, is even and there is $\vt_0\in(0,\pi/2)$
such that
$$w_*=c_K\ \text{ in }[0,\vt_0],\qquad 
w_*'>0\ \text{ in }(\vt_0,\pi/2].$$
 Moreover, $\W$ contains the set
$$\ul{\mc{W}}:=\conv\big((\ol B_{c_K}\cap\ol\O)\cup
[-w_*(\pi/2),w_*(\pi/2)]\times\{0\}\big),$$
and the inclusion is strict if $D>2d$.
\item {\em(Directions with enhanced speed).} If $D\leq2d$ then
$\vt_0=\pi/2$. Otherwise, if $D>2d$, $\vt_0<\pi/2$. Furthermore, as functions
of $D$, $\vt_0$ is strictly decreasing for $D>2d$ and $w_*(\vt)$ is strictly 
increasing if $\vt>\vt_0$. 
\end{enumerate}
\end{theorem}

If $D\leq2d$ then $\W\equiv\ol B_{c_K}\cap\ol\O$, that is, the road has no
effect on the asymptotic speed of spreading, in any direction, which means that 
the asymptotic speed coincides with the Fisher -KPP invasion speed $c_K$.  On the contrary, in the case
$D>2d$, the
spreading speed  is enhanced in all directions outside a cone around the normal to the
road. The closer the direction to the road, the higher the speed. 
Of course, $w_*(\pm\pi/2)$ coincides with $c_*$ from \eqref{e1.5}.
The opening $2\vt_0$ of this cone is explicitly given by \eqref{theta0} below.
The case $D>2d$ is summarized by Figure \ref{fig:W}.
% Let us mention the following weaker formulation for $\W$:
%
% $$\forall \zeta\in\interior(\ol\O\meno\W),\quad 
% \lim_{t\to+\infty}(u(x+ct,t),v(x+ct,y,t))=(0,0),$$
% $$\forall \zeta\in\interior(\W),\quad 
% \lim_{t\to+\infty}(u(x+ct,t),v(x+ct,y,t))=(U,V(y)),$$
% locally uniformly in $(x,y)\in\ol\Omega$.
%
\begin{figure}[ht]
\hoffset=.5cm
\begin{center}
\includegraphics[height=5cm]{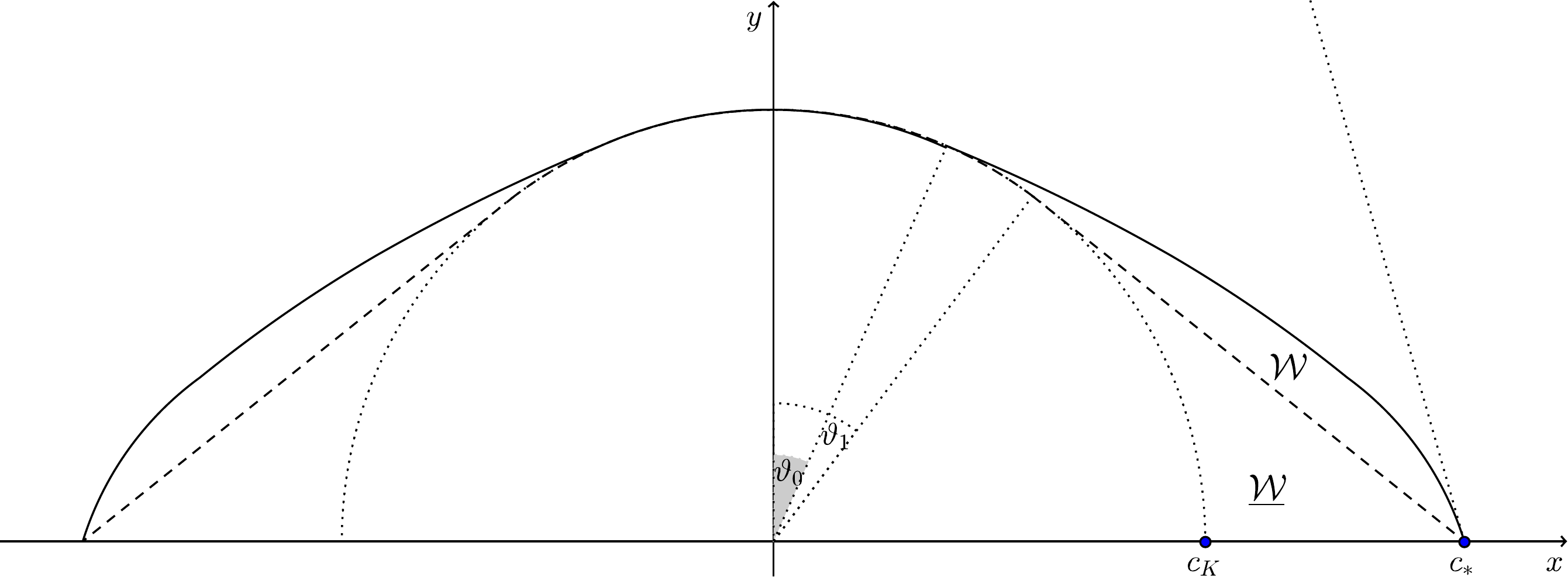}
\caption{The sets $\mc{W}$ (solid line) and $\ul{\mc{W}}$ (dashed line)
in the case $D>2d$.}
\label{fig:W}
\end{center}
\end{figure}

% The inclusion $\W\supset\ul{\W}$ stated in \thm{main} part (ii)
% is a consequence of the fact that $\W$ is convex and that it contains the
% half disc $\ol B_{c_K}\cap\ol\O$ and the segment
% $[-w_*(\pi/2),w_*(\pi/2)]\times\{0\}$.

The inclusion $\mc{W}\supset\ul{\mc{W}}$ yields the following estimates on
$\mc{W}$:
$$\vt_0<\vt_1:=\arcsin\frac{c_K}{c_*},\qquad
\forall\vt\geq\vt_1,\quad
w_*(\vt)>\frac{c_K\,c_*}
{c_K\sin\vt+\sqrt{c_*^2-c_K^2}\cos\vt}.$$
Consider now $w_*$ and $c_*$ as functions of $D$, with the other parameters
frozen. We know from \cite{BRR2} that $c_*\to\infty$ as
$D\to\infty$. Hence, the above
inequalities yield 
$$\lim_{D\to\infty}\vt_0=\lim_{D\to\infty}\vt_1=0,\qquad
\forall\vt>0,\quad\liminf_{D\to\infty}w_*(\vt)\geq \frac{c_K}{\cos\vt}.$$
Since $w_*(\vt)\leq c_K/\cos\vt$, as it is readily seen by comparison with
the tangent line $y=c_K$, we have the following
\begin{proposition}
As functions of $D$, the quantities $\vt_0$ and $w_*$ satisfy
$$\lim_{D\to\infty}\vt_0=0,\qquad 
\forall\vt\in[-\pi/2,\pi/2],\quad\lim_{D\to\infty}w_*(\vt)=
\frac{c_K}{\cos\vt}.$$
That is, as $D\nearrow\infty$, the set $\W$ increases to fill up the whole 
strip $\R\times[0,c_K)$.
\end{proposition}

Let us give an extension of Theorem \ref{thm:main}.
In \cite{BRR3}, we further investigated the effects of transport and reaction on
the road.
This results in the two additional terms $q\partial_x u$ and
$g(u)$ in the first equation of \eqref{Cauchy}.
We were able to extend the results of \cite{BRR2} under a concavity
assumption on $f$ and $g$. The additional assumption on $f$ is not required if
$g$ is a pure mortality term, i.e., $g(u)=-\rho u$ with $\rho\geq0$.
This is the most relevant case from the point of view of the applications to
population dynamics. 
The system with transport and pure mortality on the road reads
\Fi{CauchyT}
\begin{cases}
\partial_t u-D \partial_{xx} u+q\partial_x u= \nu\rest{v}{y=0}-\mu u
-\rho u, & x\in\R,\ t>0\\
\partial_t v-d\Delta v=f(v), & (x,y)\in\O,\ t>0\\
-d\partial_y\rest{v}{y=0}=\mu u-  \nu\rest{v}{y=0}, & x\in\R,\ t>0,
\end{cases}
\Ff
with $q\in\R$ and $\rho\geq0$.
The first difference with \eqref{Cauchy} is that $(\nu/\mu,1)$ is no longer a
solution if $\rho\neq0$. However, we showed in \cite{BRR3} that
\eqref{CauchyT} admits a unique positive, bounded, stationary solution
$(U_S,V_S)$, with $U_S$ constant and $V_S$ depending only on $y$ and such that
$V_S\to1$ as $y\to+\infty$. We then derived the existence of the asymptotic
speed of spreading (to $(U_S,V_S)$) in the direction of the line. This is not 
symmetric if $q\neq0$. There are indeed
two asymptotic speeds of spreading $c_*^\pm$, in the directions $\pm(1,0)$
respectively. They satisfy $c_*^\pm\geq c_K$, with strict inequality if and only 
if
\Fi{enhancement}
\di\frac{D}d>2+\frac\rho{f'(0)}\mp\frac q{\sqrt{df'(0)}}.
\Ff

The method developed in the present paper to prove \thm{main} can be adapted
to the case of system \eqref{CauchyT}.
The details on how this is achieved are given in Section \ref{sec:T} below.
In this framework, the notion of the asymptotic expansion shape is
modified by replacing $1$ with $V_S(y)$ in \eqref{W-}.
\begin{theorem}\label{thm:mainT}
For system~\eqref{CauchyT}, the following properties hold true:
\begin{enumerate}[(i)]
\item {\em(Spreading).} 
There exists an asymptotic expansion shape $\W$. 
\item {\em(Expansion shape).} 
The set $\mc{W}$ is convex and it is of the form
$$\mc{W}=\{r(\sin\vt,\cos\vt)\ :\
-\pi/2\leq\vt\leq\pi/2,\ \ 0\leq r\leq w_*(\vt)\},$$
with $w_*\in C^1([-\pi/2,\pi/2])$ such that
$$w_*=c_K\ \text{ in }[\vt_-,\vt_+],\qquad 
w_*'<0\ \text{ in }[-\pi/2,\vt_-),\qquad
w_*'>0\ \text{ in }(\vt_+,\pi/2],$$
for some critical angles $-\pi/2\leq\vt_-<0<\vt_+\leq\pi/2$.
\item {\em(Directions with enhanced speed).} If \eqref{enhancement} does not
hold then $\vt_\pm=\pm\pi/2$. Otherwise, if \eqref{enhancement} holds,
$\vt_\pm\neq\pm\pi/2$. 
% and $\vt_\pm$ is a strictly decreasing function of $D-2d$.
\end{enumerate}
\end{theorem}
\begin{figure}[ht]
\hoffset=.5cm
\begin{center}
\includegraphics[height=7cm]{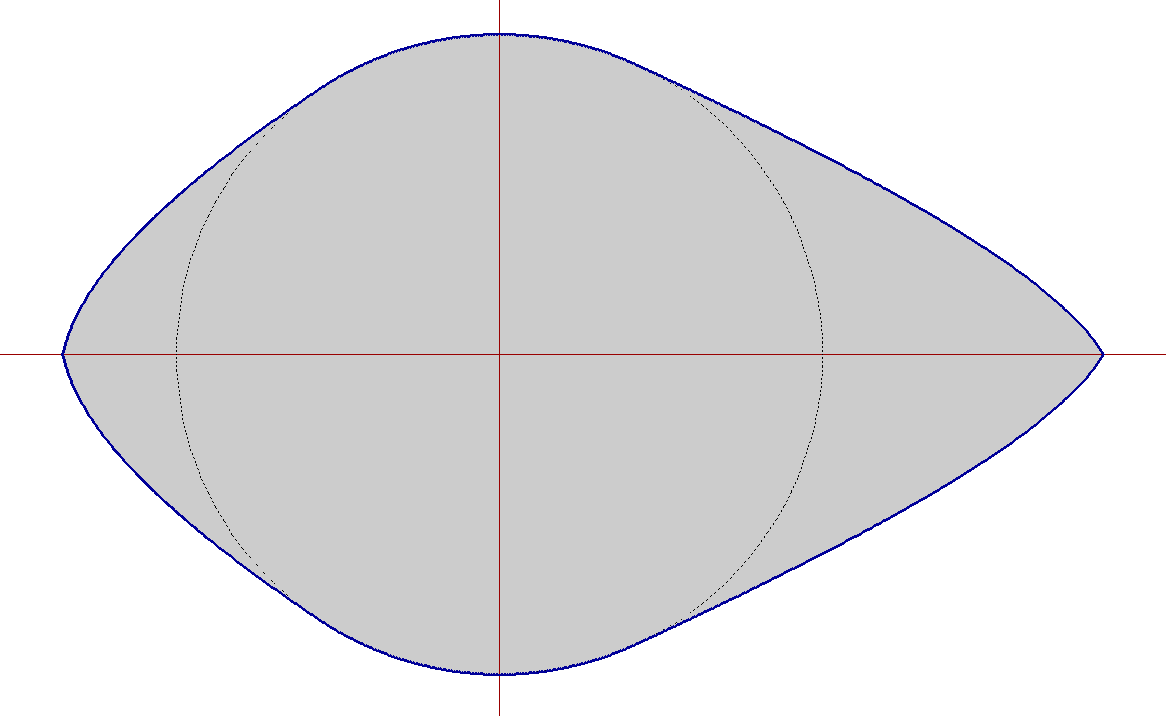}
\caption{The asymptotic expansion shape in the presence of a transport term
towards right on the road ($q>0$).}
\label{fig:Eye}
\end{center}
\end{figure}
\subsection{Discussion and comments}
Let us first comment on how the spreading velocity in the direction 
$\xi=(\sin\vt,\cos\vt)$ is sought for. 
It will be the least $c>0$ 
such that the linearisation of \eqref{Cauchy} around 0 admits solutions of the 
form
$$
(U(x,t),V(x,y,t))=(e^{-(\alpha,\beta)\.((x,0)-ct\xi)},\gamma
e^{-(\alpha,\beta)\.((x,y)-ct\xi)}),
$$
with $\alpha,\beta\in\R$, $\gamma>0$. Let us point out that $V$ is not 
exactly a planar wave in the direction $\xi$, for the simple reason that its 
level sets are not hyperplanes orthogonal to $\xi$, but to $(\alpha,\beta)$.
We will find that, when $D>2d$ and $\vt$ is larger than a critical angle $\vt_0$,
the vector $(\alpha,\beta)$ associated with the least $c$ is not parallel 
to $\xi$.
This is the reason why the velocity $w_*(\vt)$ looks different from the 
classical Freidlin-G\"artner
formula \cite{FG}, that we recall here: for a scalar equation of the form 
\begin{equation}
\label{e2.1}
u_t-\Delta u+b(x)\.\nabla u=\mu(x)u-u^2,
\end{equation}
with $\mu>0$, $\mu$ and $b$ 1-periodic, the spreading velocity in the 
direction $\xi$ is given by 
\begin{equation}
w_*(\xi)=\inf_{\xi\.\xi'>0}\frac{c_*(\xi)}{\xi\.\xi'}
\label{e2.4}
\end{equation}
where $c_*(\xi)$ is the least $c$ such that the linearisation of \eqref{e2.1} around 0:
\begin{equation}
\label{e2.2}
u_t-\Delta u+b(x)\.\nabla u=\mu(x)u,
\end{equation}
admits solutions of the form 
$$
\phi(x)e^{\lambda(x\.\xi-ct)},\  \   \   \   \   \phi>0,\ \hbox{1-periodic.}
$$
The optimal assumption for $\mu$ is not, by the way, $\mu>0$. A more general assumption is 
$\lambda_1^{per}(-\Delta-\mu(x))<0$,
where $\lambda_1^{per}$ denotes the first periodic eigenvalue.
In any case, \eqref{e2.4} gives the formula
$$\forall\xi,\xi'\in\R^N\backslash\{0\},\quad  c_*(\xi)\geq w_*(\xi)\xi\.\xi'.$$
We will see in Section 6 (Lemma \ref{lem:tangent} below) that a similar, but different, formula 
holds in our case, namely,
$$\forall \vt\in[\vt_0,\pi/2],\ \t\vt\in[0,\pi/2],\quad
w_*(\t\vt)\leq\frac{\cos(\vt-\vp_*(\vt))}{\cos(\t\vt-\vp_*(\vt))}\,w_*(\vt).
$$
It will, in fact, be derived as a consequence of the expression of the spreading velocity.

Several proofs of the Freidlin-G\"artner formula have been given, besides that 
of~\cite{FG}. See Evans-Souganidis \cite{ES1} for a viscosity 
solutions/singular perturbations approach, Weinberger \cite{W02} for an abstract 
monotone system proof; Berestycki-Hamel-Nadin \cite{BHNadin} for a PDE proof. 
See also \cite{BHN1} for equivalent formulae and estimates of the spreading 
speed in periodic media, as well as \cite{BN1} for one-dimensional general 
media.  Many of these results are explained, and developped, in \cite{BHbook}.

Let us now discuss the shape of the set ${\mc{W}}$ in Theorem \ref{thm:main}, 
and how it compares to~$\ul{\mc{W}}$. The latter  has a very natural 
interpretation as the reachable set from the origin in time~$1$ by moving with 
speed $c_*$ on the road and $c_K$ in the field.
Indeed, considering trajectories obtained by moving on the road until time 
$\lambda\in[0,1]$ and then on a straight line in the field for the remaining
time $1-\lambda$,
% Namely, subject to this constraint on the speed, one can check that the fastest
% trajectories
% from the origin to a given point are either
% segments forming an angle $\vt$ less than $\vt_1:=\arcsin(c_K/c_*)$ with the
% vertical, or an horizontal segment followed by a segment forming exactly the
% angle $\vt_1$ with the vertical.
one finds that the reachable set is
the convex hull of the union of the segment
$[-c_*,c_*]\times\{0\}$ and the half-disc
$\ol B_{c_K}\cap\ol\O$, that is, $\ul{\mc{W}}$.

Another way to obtain the set $\ul{\mc{W}}$ is the following: consider a 
set-valued map $t\mapsto U_t\in\ol\O$ and impose 
that the trace of $U_t$ expands at speed $c_*$ on the $x$-axis, and that the 
rest evolves by asking that the normal velocity of its boundary 
equals $c_K$. In PDE terms, $U_t=\{(x,y)\ :\phi(x,y,t)\geq1\}$, where $\phi$ 
solves the eikonal equation
\[\begin{cases}
\phi_t-c_K\vert\nabla\phi\vert=0 & t>0,(x,y)\in\Omega\\
\phi(x,0,t)=\1_{[-c_*t,c_*t]}(x) & t>0,x\in\R
\end{cases}\]
So, the family of sets $(U_t)_{t>0}$ is simply obtained by 
applying the Huygens principle
with the segment $[-c_*t,c_*t]$ on the road as a source.  
In other words, $t\ul\W=U_t$ and it evolves with normal 
velocity $c_K$. Notice that imposing that a family of sets $(tA)_{t>0}$ evolves 
with normal velocity $c_K$ forces the curvature of $A$ to be either $1/c_K$ or 
$0$, i.e.,~$A$ is locally either a disc of radius $c_K$ or a half-plane.
It would have been tempting to think that ${\mc{W}}$   coincides with 
$\ul{\mc{W}}$, just as in the singular perturbation approach to front propagation 
in parabolic equations or systems - see Evans-Souganidis \cite{ES1}, \cite{ES2}.
The fact that the asymptotic expansion shape is actually larger than this set
is remarkable. And, as a matter of fact, we estimate in Proposition 
\ref{pro:normalspeed} below
the difference between ${\mc{W}}$ and $\ul{\mc{W}}$ in terms of the normal velocities of
their boundaries when magnified by $t$. Namely, we discover that the normal 
speed of $(t\W)_{t>0}$ at a boundary point $t(\sin\vt,\cos\vt)$, 
$\vt>\vt_0$, coincides with the normal speed of the level lines of
the planar wave for the linearised system which defines $w_*(\vt)$ 
- see the next section. 
This speed is larger than $c_K$ because the exponential decay rate of $V$ in 
the direction orthogonal to its level lines is less than the critical one: 
$\sqrt{f'(0)/d}$. 
We expect this decay to be approximatively satisfied for large time by the 
solution of \eqref{Cauchy} emerging from a compactly supported initial datum. 
Thus, heuristically, the presence of the road would result in an ``unnatural'' 
decay for solutions of the KPP equation with compactly supported initial data, 
which, in turns, would be the reason why $\W$ does not coincide with the set 
$\ul\W$ following from Huygens' principle.

%%%%%%%%%%%%%%%%%%%%%%%%%%%%%%%%%%%%%%%%%%%%%%%%%%%%%%%%%%%%%%%%%%

%%%%%%%%%%%%%%%%%%%%%%%%%%%%%%%%%%%%%%%%%%%%%%%%%%%%%%%%%%%%%%%%%%

\section{Planar waves for the linearised system}
\label{sec:exp}

Consider the linearisation of system \eqref{Cauchy} around $v=0$:
\Fi{linear}
\begin{cases}
\partial_t u-D \partial_{xx} u= \nu\rest{v}{y=0}-\mu u &
x\in\R,\
t>0\\
\partial_t v-d\Delta v=f'(0)v & (x,y)\in\O,\ t>0\\
-d\partial_y\rest{v}{y=0}=\mu u(x,t)-  \nu\rest{v}{y=0} & x\in\R,\ t>0.
\end{cases}
\Ff
% Up to rescaling the parameters, we can take $\nu=1$.
Take a unit vector $\xi=(\xi_1,\xi_2)$, with $\xi_2\geq0$. By symmetry, we
restrict to $\xi_1\geq0$. As said above, solutions are sought for in the form
\begin{equation}
\label{exp}
(U(x,t),V(x,y,t))=(e^{-(\alpha,\beta)\.((x,0)-ct\xi)},\gamma
e^{-(\alpha,\beta)\.((x,y)-ct\xi)}),
\end{equation}
with $c\geq0$, $\gamma>0$ and $\alpha,\beta\in\R$ (not
necessarily positive). This leads to the system
$$
\begin{cases}
c\,\xi\.(\alpha,\beta)-D\alpha^2=\nu\gamma-\mu\\
c\,\xi\.(\alpha,\beta)-d(\alpha^2+\beta^2)=f'(0)\\
d\gamma\beta=\mu-\nu\gamma.
\end{cases}
$$
The third equation yields $\gamma=\mu/(\nu+d\beta)$ and then
$\beta>-\nu/d$. Setting $\chi(s):=\mu s/(\nu+s)$, the system on $(\alpha,\beta)$
reads
\begin{equation}\label{algebraic}
\begin{cases}
c\xi_1\alpha+c\xi_2\beta-D\alpha^2=-\chi(d\beta)\\
c\xi_1\alpha+c\xi_2\beta-d(\alpha^2+\beta^2)=f'(0).
\end{cases}
\end{equation}
The first equation in the unknown $\alpha$ has the roots
$$\alpha_D^\pm(c,\beta):=\frac1{2D}\left(c\xi_1\pm\sqrt{(c\xi_1)^2+
4D\left(c\xi_2\beta+\chi(d\beta)\right)}\right),$$
which are real if and only if $\beta$ is larger than some value 
$\ul\beta(c)\in(-\nu/d,0]$.
The set of real solutions
$(\beta,\alpha)$ of the first equation in \eqref{algebraic} is then
$\;\Sigma(c)=\Sigma^-(c)\cup\Sigma^+(c)$, with
$$\Sigma^\pm(c):=\{(\beta,\alpha_D^\pm(c,\beta))\ :\
\beta\geq\ul\beta(c)\}.$$
This is a smooth curve with leftmost
point $(\ul\beta(c),c\xi_1/2D)$. Rewriting the second equation in
\eqref{algebraic} as
$\displaystyle|(\alpha,\beta)-\frac c{2d}\xi|^2=\frac{c^2}{4d^2}-\frac{f'(0)}d$,
we see that it has solution if and only if $c\geq c_K$, where 
$c_K:=2\sqrt{df'(0)}$ is the
invasion speed in the field. 
In the $(\beta,\alpha)$ plane, it represents the circle $\Gamma(c)$ of centre
$C(c)$ and radius $r(c)$ given by
$$C(c)=\frac c{2d}(\xi_2,\xi_1),\qquad r(c)=\frac{\sqrt{c^2-c_K^2}}{2d}.$$
Let $\mc{S}(c)$ denote the closed set bounded from below by $\Sigma^-(c)$ and
from above by
$\Sigma^+(c)$ and let $\mc{G}(c)$ denote the closed disc
with boundary $\Gamma(c)$.
Exponential functions of the type \eqref{exp} are supersolutions of
\eqref{linear} if and only if $(\beta,\alpha)\in \mc{S}(c)\cap
\mc{G}(c)$. Since the centre $C(c)$ belongs to the line
$s\mapsto s(\xi_2,\xi_1)$ and the closest point of $\Gamma(c)$ to the origin,
$P(c):=C(c)-r(c)(\xi_2,\xi_1)$, satisfies
$$P'(c)\.(\xi_2,\xi_1)=\frac1{2d}-\frac c{d\sqrt{c^2-c_K^2}}<0,\qquad
\lim_{c\to+\infty}P(c)=0,$$
we find that 
$$\forall c'\geq c\geq c_K,\quad \mc{G}(c')\supset\mc{G}(c),\qquad
\bigcup_{c\geq c_K}\mc{G}(c)=\{(\beta,\alpha)\ :\
(\beta,\alpha)\.(\xi_2,\xi_1)>0\}.$$
% 
% As $c$ increases, the radius of $\Gamma(c)$ increases and its bottom
% $(B_1(c),B_2(c))$ moves
% rightward and downward, since
% $$B_1'(c)\geq0,\qquad
% B_2'(c)=\frac1{2d}\left(\xi_1-\frac c{\sqrt{c^2-c_K^2}}\right)<0.$$ 
% In addition, as $c\to\infty$, $B_2(c)$ tends to $0$ or $-\infty$ depending on
% the fact that $\xi_2=0$ or not.
On the other hand, $\alpha_D^+(c,\beta)$ is increasing in $c$ and concave
in $\beta$, the latter following from the concavity of
$c\xi_2\beta+\chi(d\beta)$.                     
 \begin{figure}[ht]
 \centering
 \subfigure[$c<w_*$]
   {\includegraphics[width=4.5cm]{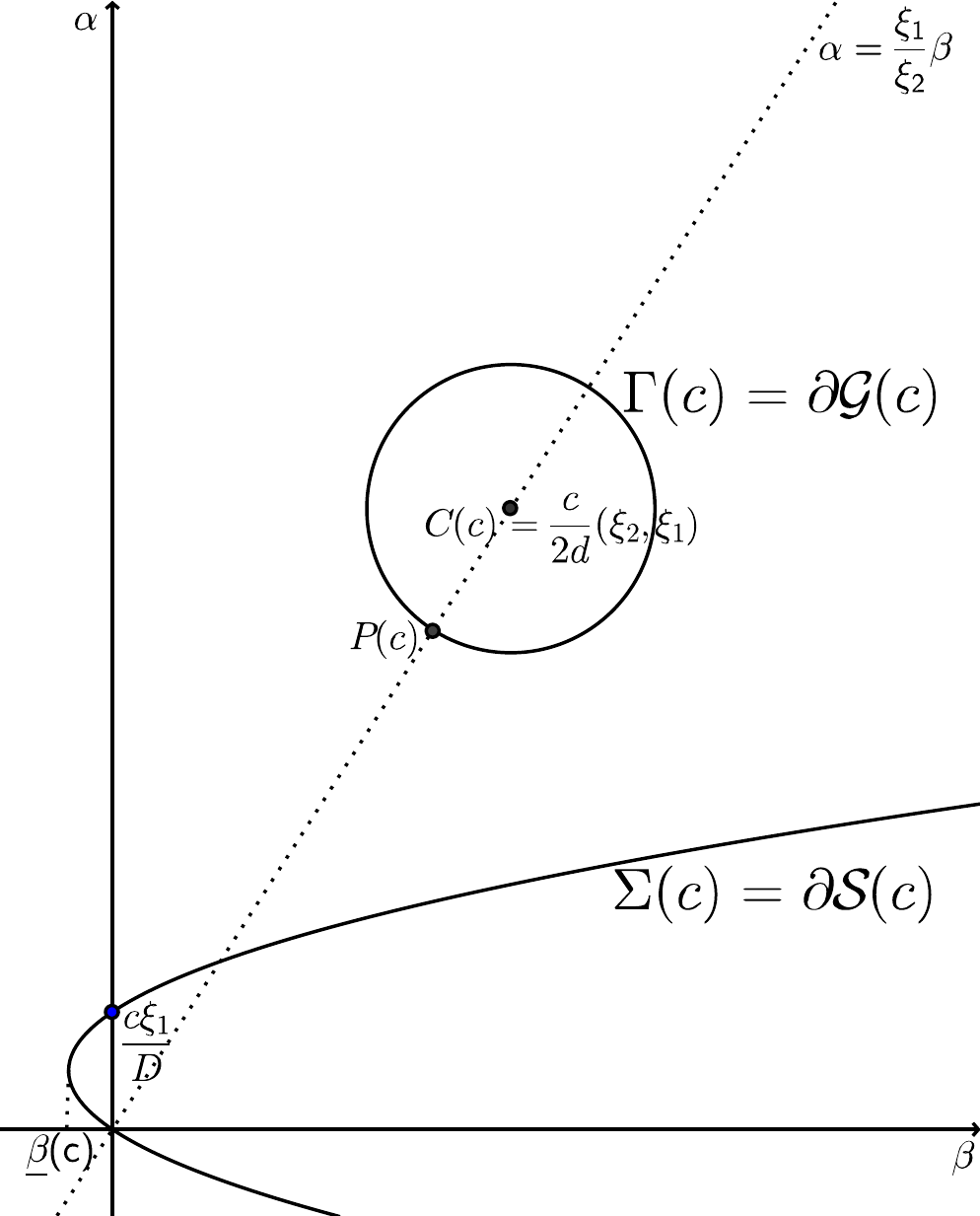}}
 \hspace{3mm}
 \subfigure[$c=w_*$]
   {\includegraphics[width=4.5cm]{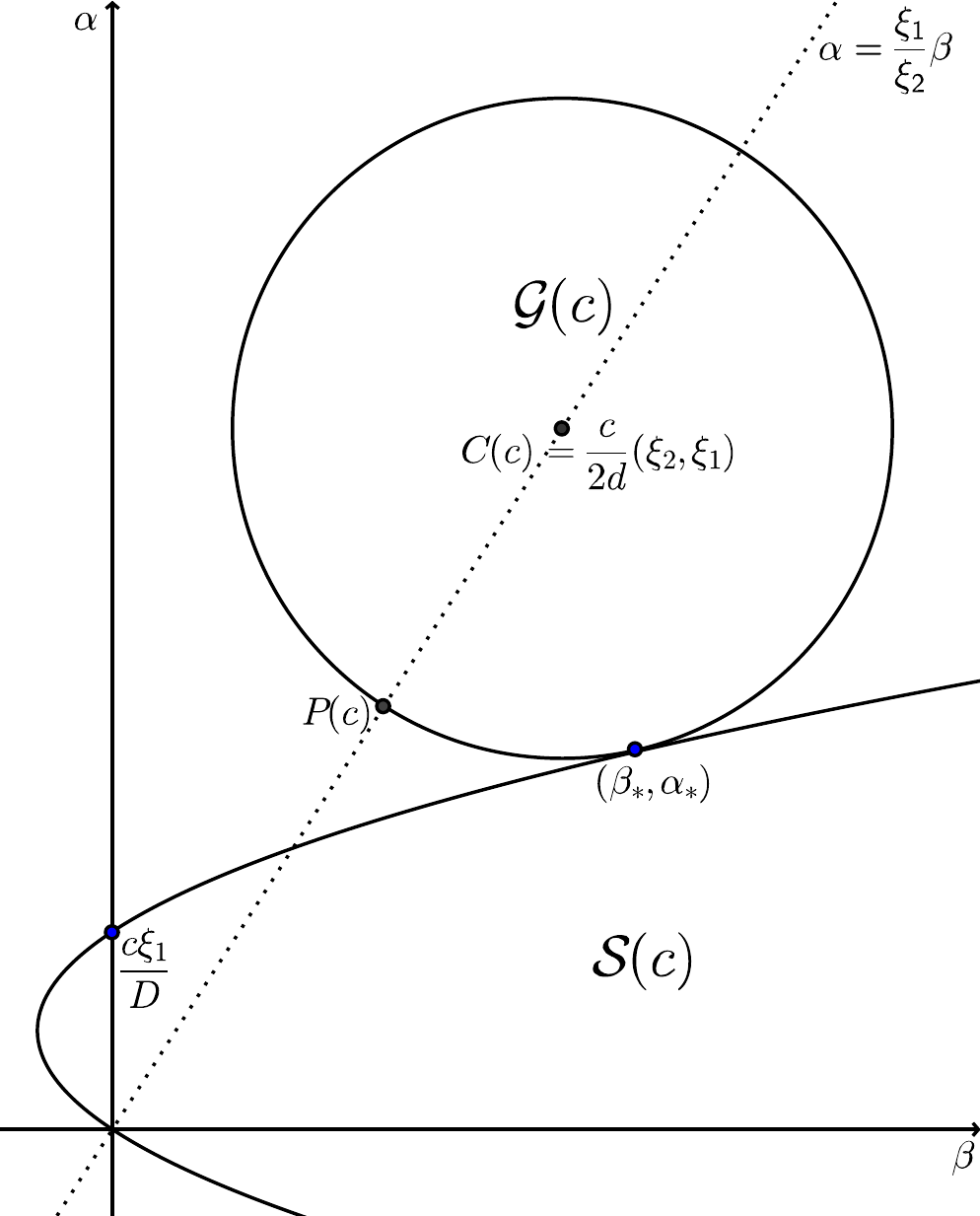}}
 \hspace{3mm}
 \subfigure[$c>w_*$]
   {\includegraphics[width=4.5cm]{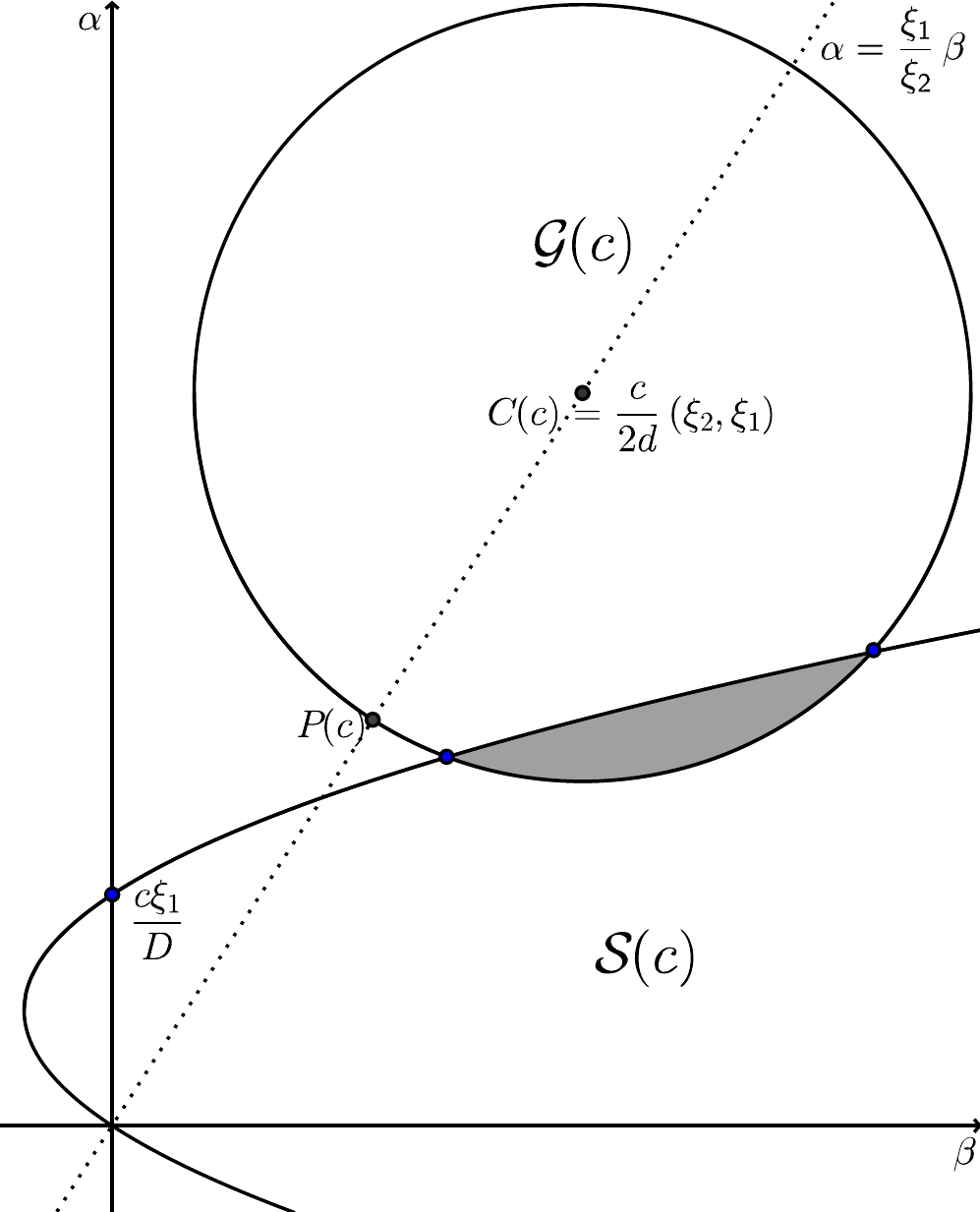}}
 \caption{The case $w_*>c_K$; supersolutions correspond to the shaded region.}
 \label{fig:w*>cK}
%  \end{figure}
% \begin{figure}[ht]
\vspace{5mm}
\centering
 \subfigure[$c=c_K$]
   {\includegraphics[width=4.5cm]{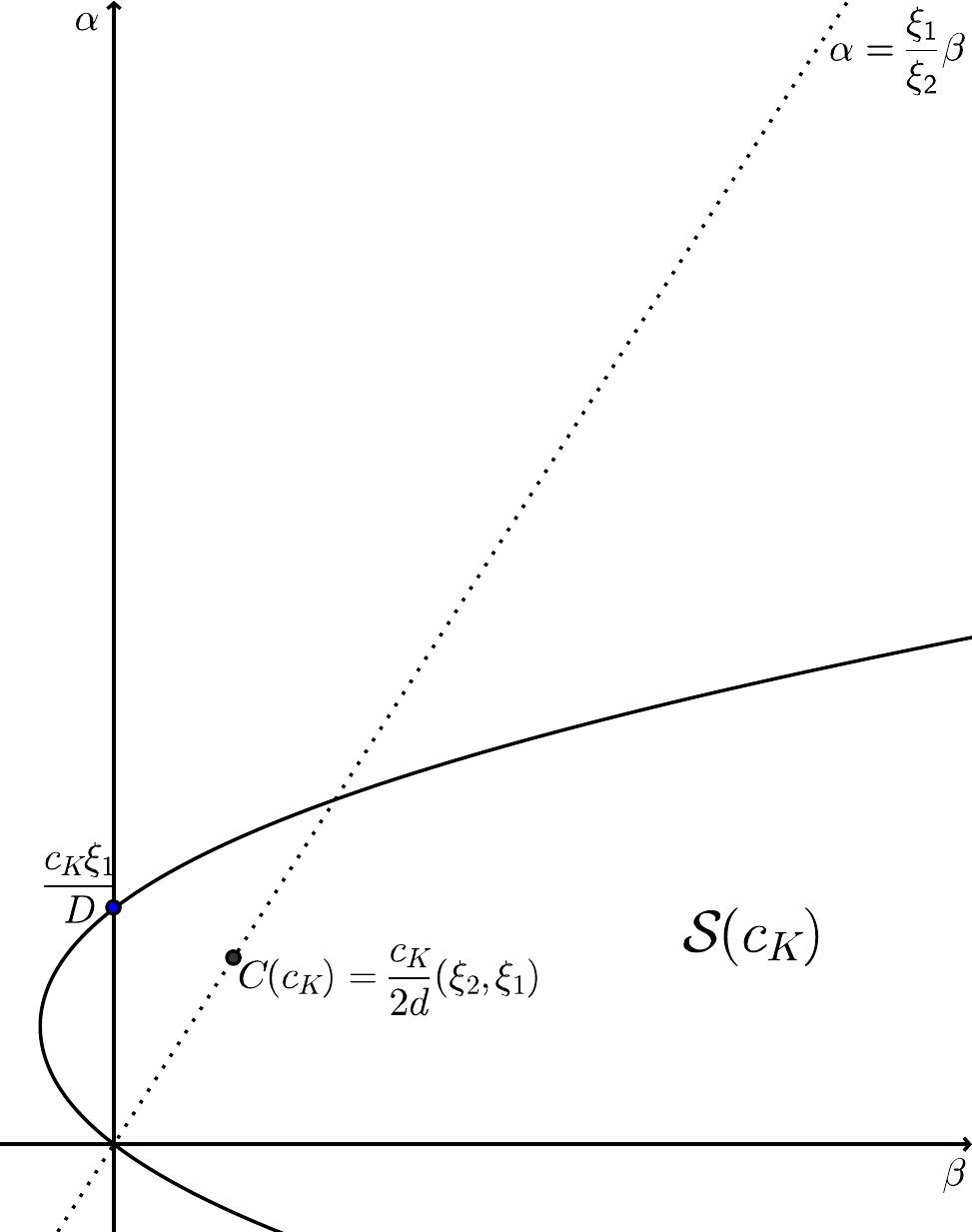}}
 \hspace{20mm}
 \subfigure[$c>c_K$]
   {\includegraphics[width=4.5cm]{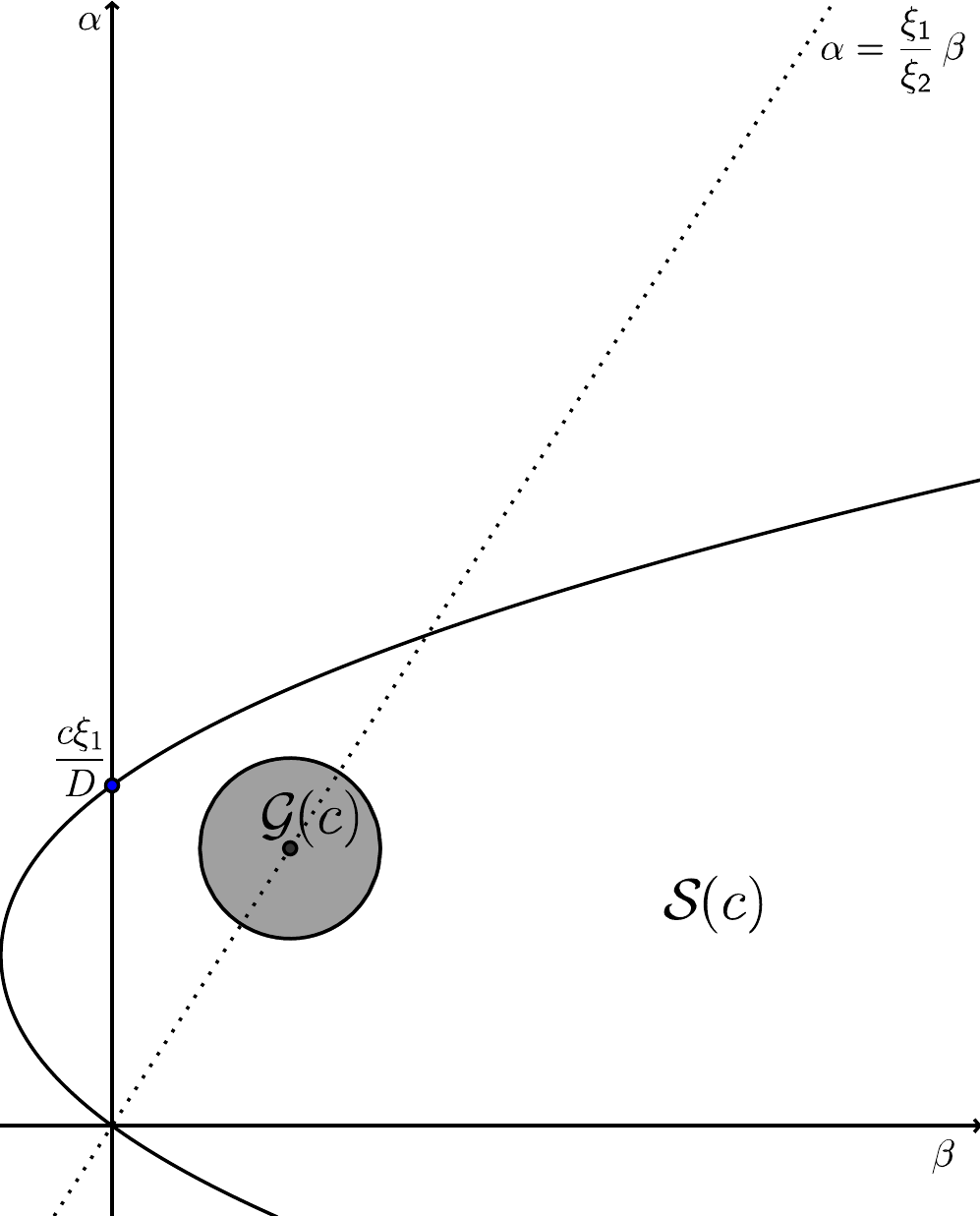}}
 \caption{The case $w_*=c_K$; supersolutions correspond to the shaded region.}
 \label{fig:w*=cK}
 \end{figure}

\noindent Therefore, there exists $w_*\geq c_K$, depending on $\xi$, such
that 
$$\mc{S}(c)\cap\mc{G}(c)\neq\emptyset\quad\sse\quad c\geq w_*,$$
with $\mc{S}(w_*)\cap\mc{G}(w_*)$ consisting in a singleton, denoted by
$(\beta_*,\alpha_*)$, see Figures~\ref{fig:w*>cK} and~\ref{fig:w*=cK}. 
Moreover,
$w_*=c_K$ if and only if $C(c_K)\in\mc{S}(c_K)$, namely, if and only if $C(c_K)$ 
satisfies the first
condition in \eqref{algebraic} with $=$ replaced by $\geq$ :
$$\frac{c_K^2}{2d}-\frac{Dc_K^2}{4d^2}\xi_1^2\geq
-\frac{\mu c_K\xi_2}{2\nu+c_K\xi_2}.$$
Since $\xi_1^2=1-\xi_2^2$, this inequality rewrites
$$
2d+D(\xi_2^2-1)+\frac{4d^2\mu\xi_2}{2\nu c_K+c_K^2\xi_2}\geq0.
$$
The function $\Phi:[0,+\infty)\to\R$ defined by
$$\Phi(s):=2d+D(s^2-1)+\frac{4d^2\mu s}{2\nu c_K+c_K^2s},$$ is increasing
and satisfies $\Phi(0)=2d-D$, $\Phi(1)>0$. As a consequence,
$w_*=c_K$ if and only if either $D\leq2d$, or $D>2d$ and 
$\xi_2\geq\Phi^{-1}(0)\in(0,1)$. Observe that the sets $\mc{S}(c)$ shrink as 
$D$ increases and therefore $w_*$ is a strictly increasing function of $D$ when 
$w_*>c_K$. 
% We call $(\beta_*,\alpha_*)$ the unique point
% belonging to $\mc{S}(w_*)\cap\mc{G}(w_*)$. It reduces to the
% centre $C(c_K)$ if $w_*=c_K$, otherwise it is the unique tangent point
% between $\mc{S}(w_*)$ and $\mc{G}(w_*)$. 

We now consider the critical speed $w_*$ as a function of the angle formed by
the vector $\xi$ and the vertical axis. Namely, for $\vt\in[-\pi/2,\pi/2]$,
we call $w_*(\vt)$ the quantity defined above associated with
$\xi=(\sin\vt,\cos\vt)$. 
We further let $(\beta_*(\vt),\alpha_*(\vt))$ denote the
first contact point $(\beta_*,\alpha_*)$.
For $\vt=\pi/2$, the above construction reduces exactly
to the one of \cite{BRR2}, thus $w_*(\pi/2)$ coincides with the value $c_*$ 
arising in \eqref{e1.5}.
The function $w_*$ is even and continuous, as it is immediate to verify.
We know that if $D\leq2d$ then $w_*\equiv c_K$. Otherwise, if $D>2d$,
$w_*(\vt)>c_K$ if and only if $\vt>\vt_0$, where
\Fi{theta0}
\vt_0:=\arccos(\Phi^{-1}(0)).
\Ff
Notice that $\vt_0$ is a decreasing function of $D$.
We finally define
$$\mc{W}:=\{r(\sin\vt,\cos\vt)\ :\
-\pi/2\leq\vt\leq\pi/2,\ \ 0\leq r\leq w_*(\vt)\}.$$
% Suppose that $\xi_2<\Phi^{-1}(D-2d)$, whence $w_*(\xi)>c_K$. The bifurcation
% analysis of \cite{BRR2} yields, for $c<w_*(\xi)$ close
% enough to $w_*(\xi)$, the existence of an exponential solution of
% the linearised system with $\alpha,\beta,\gamma\in\C\backslash\R$.
The object of Sections \ref{sec:subsolutions} and \ref{sec:spreading} is to
show that $\mc{W}$ is the asymptotic expansion set for \eqref{Cauchy}.

%%%%%%%%%%%%%%%%%%%%%%%%%%%%%%%%%%%%%%%%%%%%%%%%%%%%%%%%%%%%%%%%%%

\section{Compactly supported subsolutions}
\label{sec:subsolutions}

This section is dedicated to the construction, for all $\vt\in(-\pi/2,\pi/2)$,
of
compactly supported
subsolutions moving in the direction $\xi=(\sin\vt,\cos\vt)$ with speed
less than, but arbitrarily close to, $w_*(\vt)$. 
We derive the following
\begin{lemma}\label{lem:gensub}
For all $\vt\in(-\pi/2,\pi/2)$ and $\e>0$, there exist $c>w_*(\vt)-\e$ and a
pair $(\ul u,\ul v)$ of nonnegative functions with the following properties:
$\rest{\ul u}{t=0}$
and $\rest{\ul v}{t=0}$ are compactly supported,
\Fi{sub>0}
\exists(\hat x,\hat y)\in\ol\O,\quad \forall t\geq0,\quad
 \ul v(\hat x+ct\sin\vt,\hat y+ct\cos\vt,t)=\ul v(\hat x,\hat y,0)>0,
\Ff
and $\kappa (\ul
u,\ul v)$ is a generalised subsolution of \eqref{Cauchy} for all 
$\kappa\in(0,1]$.
% $\rest{v}{t=0}\not\equiv0$ and
% $$\forall(x,y)\in\ol\O,\ t\geq0,\quad
% \ul v(x+ct\sin\vt,y+ct\cos\vt,t)=\ul v(x,y,0).$$
%
% Moreover, if $|\vt|>\vt_0$, we can take $(\hat x,\hat y)=(0,0)$.
\end{lemma}

By symmetry, it is sufficient to prove the lemma for $\vt\geq0$.
The case $\vt=\pi/2$ was treated in \cite{BRR2}.
If $\vt\in[0,\vt_0]$ then $w_*(\vt)=c_K$ and the construction
is standard, as we will see in Section \ref{sec:vertical}.
In Section \ref{sec:horizontal} we treat the remaining cases by exploiting the
analysis of planar waves
performed in the previous section. We will proceed as follows:
\begin{enumerate}
\item We first give a definition of generalised subsolutions adapted to our context.

 \item For $c\in(0,w_*(\vt))$ close enough to $w_*(\vt)$, we apply Rouch\'e's 
theorem to prove the existence of a complex exponential solution 
$(U,V)$ of the linearised system, which
moves in the direction $\xi=(\sin\vt,\cos\vt)$ with speed $c$. 
We actually work on a perturbed system
in order to get strict subsolutions of the nonlinear one.

 \item The connected components of the positivity set of 
$u:=\Re U$ are bounded intervals and those of $v:=\Re V$ are infinite strips.
In order to truncate those strips, we consider the reflection $v^L$ of $v$
with respect to the line $(x,y)\.\xi^\perp=L>0$. We then define the pair
$(\ul u,\ul v)$ by setting $(\ul u,\ul v)=(u,v-v^L)$ in a connected component of
the positivity sets of $u$ and $v-v^L$, $(0,0)$ outside.

 \item The function $\ul v$ is automatically a generalised subsolution of the
equation in the
field. We show that, choosing $L$ large enough, $(\ul u,\ul v)$ is a generalised
subsolution of the equations on the road too.
\end{enumerate}

%%%%%%%%%%%%%%%%%%%%%%%%%%%%%%%%%%%%%%%%%%%%%%%%%%%%%%%%%%%%%%%%%%
\subsection{Sub/supersolutions}
\label{sec:preliminar}
In the sequel, we will need to compare the solution of the Cauchy problem with 
a pair $(\ul u,\ul v)$ which is a subsolution inside some regions, vanishes 
on their boundaries, and is truncated to 0 outside. In the case of a single 
equation, such type of functions are {\em generalised subsolutions}, in the 
sense that they satisfy the comparison principle with supersolutions. This 
kind of properties has the flavour of those presented in~\cite{BL}.
In the case of a system, this property may not hold because, roughly speaking, 
one could truncate one component in a region where it is needed for the others 
to be subsolutions. This is why we need a different notion of generalised 
subsolution.

We consider pairs $(\ul u,\ul v)$ such that
$\ul u$ is the maximum of subsolutions of the first equation in \eqref{Cauchy}
with $v=\ul v$, while $\ul v$ is the maximum of subsolutions of the
second equation and of the last equation with $u=\ul u$. More precisely:

\begin{definition}\label{def:gensub}
A pair $(\ul u,\ul v)$ is a {\em generalised subsolution} of \eqref{Cauchy} if
$\ul u$, $\ul v$ are continuous and satisfy the following properties:
\begin{enumerate}[(i)]
 \item for any $x\in\R$, $t>0$, there is a
function $u$ such that $u\leq\ul u$ in a neighbourhood of $(x,t)$ and, at
$(x,t)$ (in the classical sense),
$$u=\ul u,\qquad\partial_t u-D \partial_{xx} u+\mu u\leq\nu 
\rest{\ul v}{y=0};$$
 \item for any $(x,y)\in\ol\O$, $t>0$, there is a function $v$ such that
$v\leq\ul v$ in a neighbourhood of $(x,y,t)$ and, at $(x,y,t)$,
$$v=\ul v,\qquad\partial_t v-d\Delta v\leq f(v)\ \text{ if }y>0,\quad
-d\partial_y v+\nu v\leq\mu\ul u\ \text{\ if }y=0.$$
\end{enumerate}
\end{definition}

Although this will not be needed in the paper, we may define generalised supersolutions  in analogous way, by
replacing ``$\leq$'' with ``$\geq$'' everywhere in Definition
\ref{def:gensub}. This notion is stronger than that of
viscosity solution (see, e.g., \cite{user}).
Nevertheless, it recovers: (i) classical subsolutions, (ii) maxima of
classical subsolutions and (iii) generalised subsolutions in the sense of
\cite{BRR2}.
 From now on, generalised sub and supersolutions are understood
in the sense of Definition \ref{def:gensub}. The comparison principle reads:

\begin{proposition}\label{pro:gencom}
 Let $(\ul u,\ul v)$ and $(\ol u,\ol v)$ be respectively a generalised
subsolution bounded from above and a generalised supersolution bounded from
below of \eqref{Cauchy} such that $(\ul u,\ul v)$ is below $(\ol u,\ol v)$ at
time $t=0$. Then $(\ul u,\ul v)$ is below $(\ol u,\ol v)$ for all $t>0$.
\end{proposition}

The proof is similar to the one of Proposition 3.3 in
\cite{BRR2}, even if the notion of sub and supersolution is slightly more
general here. It is included here in Appendix \ref{sec:gencom} for
the sake of completeness.

\subsection{The case $\vt\leq\vt_0$}
\label{sec:vertical}

Let $\lambda(R)$ and $\vp$ be the principal eigenvalue and eigenfunction of the
operator
$-d\Delta-c(\sin\vt,\cos\vt)\.\nabla$ in the two dimensional ball $B_R$, with
Dirichlet boundary condition. This operator can be reduced to a self-adjoint one
by multiplying the functions times 
$e^{(\sin\vt,\cos\vt)\.(x,y)c/2d}$. One then finds that
$(\lambda(R)-c^2/4d)/d$ is equal to the \pe\ of $-\Delta$ in $B_R$. Whence, for
$0<c<w_*(\vt)=c_K$,
$$\lim_{R\to\infty}\lambda(R)=\frac{c^2}{4d}<f'(0).$$
% 
% 
% It is well known that,  
% less than $f'(0)>0$ for $R$ sufficiently large (one reduces to the principal
% eigenvalue of
% the Laplacian by replacing
% $\vp$ with
There is then $R>0$ such that $f(s)\geq\lambda(R)s$ for $s>0$
small enough, and therefore we can normalise the principal eigenfunction $\vp$
in such a way that
$$\forall\kappa\in[0,1],\quad
-d\Delta(\kappa\vp)-c(\sin\vt,\cos\vt)\.\nabla(\kappa\vp)\leq 
f(\kappa\vp)\quad\text{in }B_R.$$
It follows that the pair $(\ul u,\ul v)$ defined by $\ul u\equiv0$,
$$\ul v(x,y,t)=
\begin{cases}
 \vp(x-ct\sin\vt,y-R-ct\cos\vt) & \text{if }(x,y-R)-ct(\sin\vt,\cos\vt)\in B_R\\
 0 & \text{otherwise}
\end{cases}
$$
satisfies the properties stated in Lemma \ref{lem:gensub}.

%%%%%%%%%%%%%%%%%%%%%%%%%%%%%%%%%%%%%%%%%%%%%%%%%%%%%%%%%%%%%%%%%%

\subsection{The case $\vt>\vt_0$}
\label{sec:horizontal}

Suppose now that $D>2d$ and consider $\vt\in(\vt_0,\pi/2)$. Call
$$\xi:=(\sin\vt,\cos\vt),\qquad\xi^\perp:=(-\cos\vt,\sin\vt),$$
and, to ease notation, $w_*=w_*(\vt)$, $\alpha_*=\alpha_*(\vt)$,
$\beta_*=\beta_*(\vt)$.

%%%%%%%%%%%%%%%%%%%%%%%%%%%%%%%%%%%%%%%%%%%%%%%%%%%%%%%%%%%%%%%%%%

\subsubsection{Complex exponential solutions for the penalised system}

We start with the following
\begin{lemma}\label{lem:C-subsolution}
For $c\in(0,w_*)$ close enough to $w_*$, \eqref{linear} admits an
exponential solution $(U,V)$ of
the type \eqref{exp} with $\alpha,\beta,\gamma\in\C\backslash\R$ satisfying
\Fi{slopes}
\Re\alpha,\Re\beta>0,\qquad
0<\frac{\Im\alpha}{\Im\beta}<\frac{\Re\alpha}{\Re\beta}<\frac{\xi_1}{\xi_2}.
\Ff
\end{lemma}

\begin{proof}
For $c<w_*$, problem \eqref{linear} does not admit exponential solutions of
the type \eqref{exp}, with $\alpha,\beta,\gamma\in\R$. However, if $w_*-c$
is small enough, applying the Rouch\'e theorem to the distance between
$\Gamma$ and $\Sigma$ as a function of $\beta$, one obtains an
exponential solution $(U,V)$ with
$\alpha,\beta,\gamma\in\C$, depending on $c$, and satisfying
$$\alpha=\alpha_r+i\alpha_i,\qquad
\beta=\beta_r+i\beta_i,\qquad
\gamma=\frac\mu{\nu+d\beta},$$
$$\beta_r=\beta_*+O(w_*-c),\qquad
0\neq\beta_i=O(\sqrt{w_*-c}).$$
See the proof of Lemma 6.1 in \cite{BRR2} for the details.
Writing separately the real and complex terms of the second equation of the
system \eqref{algebraic} satisfied by $\alpha,\beta$, we get
\Fi{algebraic-C}
\begin{cases}
   c\xi\.(\alpha_r,\beta_r)-d(\alpha_r^2-\alpha_i^2+\beta_r^2-\beta_i^2)=f'(0)\\
   c\xi\.(\alpha_i,\beta_i)-2d(\alpha_r\alpha_i+\beta_r\beta_i)=0.
  \end{cases}
\Ff
The first equation of \eqref{algebraic-C} yields
$$c\xi_1\alpha_r-d(\alpha_r^2-\alpha_i^2)=
f'(0)-[c\xi_2\beta_*-d\beta_*^2]
+o(1)=c\xi_1\alpha_*-d\alpha_*^2+o(1),\quad\text{as }c\to w_*.$$
It follows that
$$\liminf_{c\to w_*^-}(\alpha_r-\frac c{2d}\xi_1)^2\geq
\liminf_{c\to w_*^-}(\alpha_*-\frac c{2d}\xi_1)^2.$$
In particular, $\alpha_r$ stays away from $\frac c{2d}\xi_1$ as $c\to w_*^-$.
Rewriting the second equation of~\eqref{algebraic-C} as 
$(\frac c{2d}\xi_1-\alpha_r)\alpha_i=(\beta_r-\frac c{2d}\xi_2)\beta_i$,
we then infer that $\alpha_i=O(\sqrt{w_*-c})$. Then, since
$\Im\gamma=O(\sqrt{w_*-c})$ too, considering the real part of
\eqref{algebraic}, we eventually find that $\alpha_r=\alpha_*+o(1)$ as
$c\to w_*^-$.

We use again the second equation ofprecisely\eqref{algebraic-C} to derive
$$\lim_{c\to w_*^-}\frac{\alpha_i}{\beta_i}=
\lim_{c\to w_*^-}\frac{\beta_r-\frac{c}{2d}\xi_2}
{\frac{c}{2d}\xi_1-\alpha_r}=\frac{\beta_*-\frac{c}{2d}\xi_2}
{\frac{c}{2d}\xi_1-\alpha_*}.$$
The latter represents the slope of the tangent line to $\mc{G}(w_*)$ at the
point $(\beta_*,\alpha_*)$.
From the convexity of $\mc{S}(w_*)$ we know that this line intersects the
$\alpha$-axis at some $\alpha>c\xi_1/D$. It follows in particular that its
slope is smaller than the one of the line through $(0,0)$ and
$(\beta_*,\alpha_*)$. This, in turn, is less than the slope of the line
through $(0,0)$ and the centre of $\mc{G}(w_*)$, which is parallel to
$(\xi_2,\xi_1)$, see Figure \ref{fig:slopes} (a).
%  \begin{figure}
%  \centering
%  \includegraphics[width=9cm]{slopesX.eps}
%  \caption{Slope of the tangent line at $(\beta_*,\alpha_*)$.}
%  \label{fig:slopes}
%  \end{figure}
We deduce that
$$0<\lim_{c\to w_*^-}\frac{\alpha_i}{\beta_i}<
\frac{\alpha_*}{\beta_*}=\lim_{c\to
w_*^-}\frac{\alpha_r}{\beta_r}<\frac{\xi_1}{\xi_2}.$$
This concludes the proof.
\end{proof}
 \begin{figure}[ht]
 \centering
 \subfigure[Slope of the tangent line at $(\beta_*,\alpha_*)$]
   {\includegraphics[width=6cm]{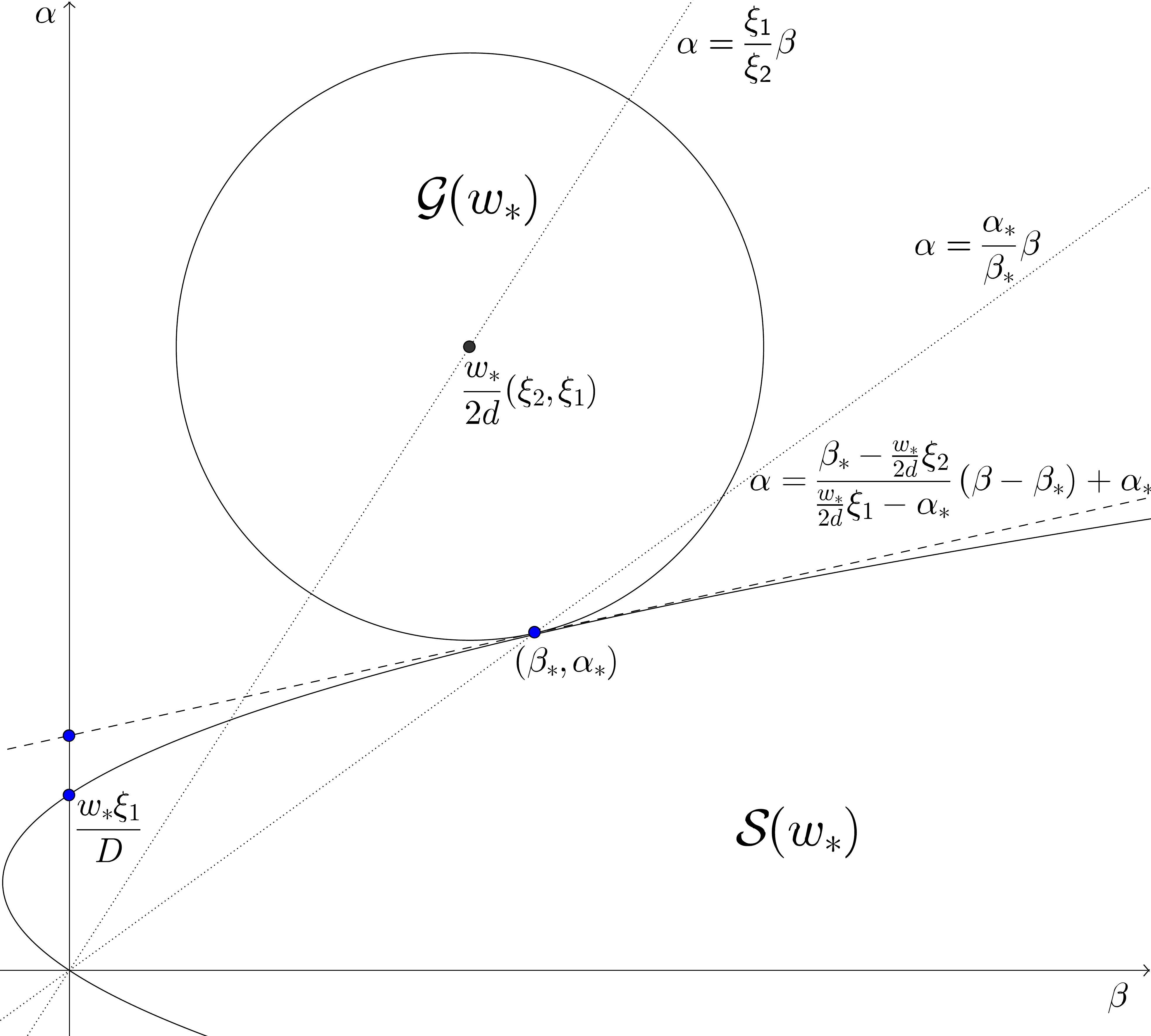}} 
 \hspace{10mm}
 \subfigure[The sets $\mc{V}_t$ and $\mc{V}^L_t$]
   {\includegraphics[width=7cm]{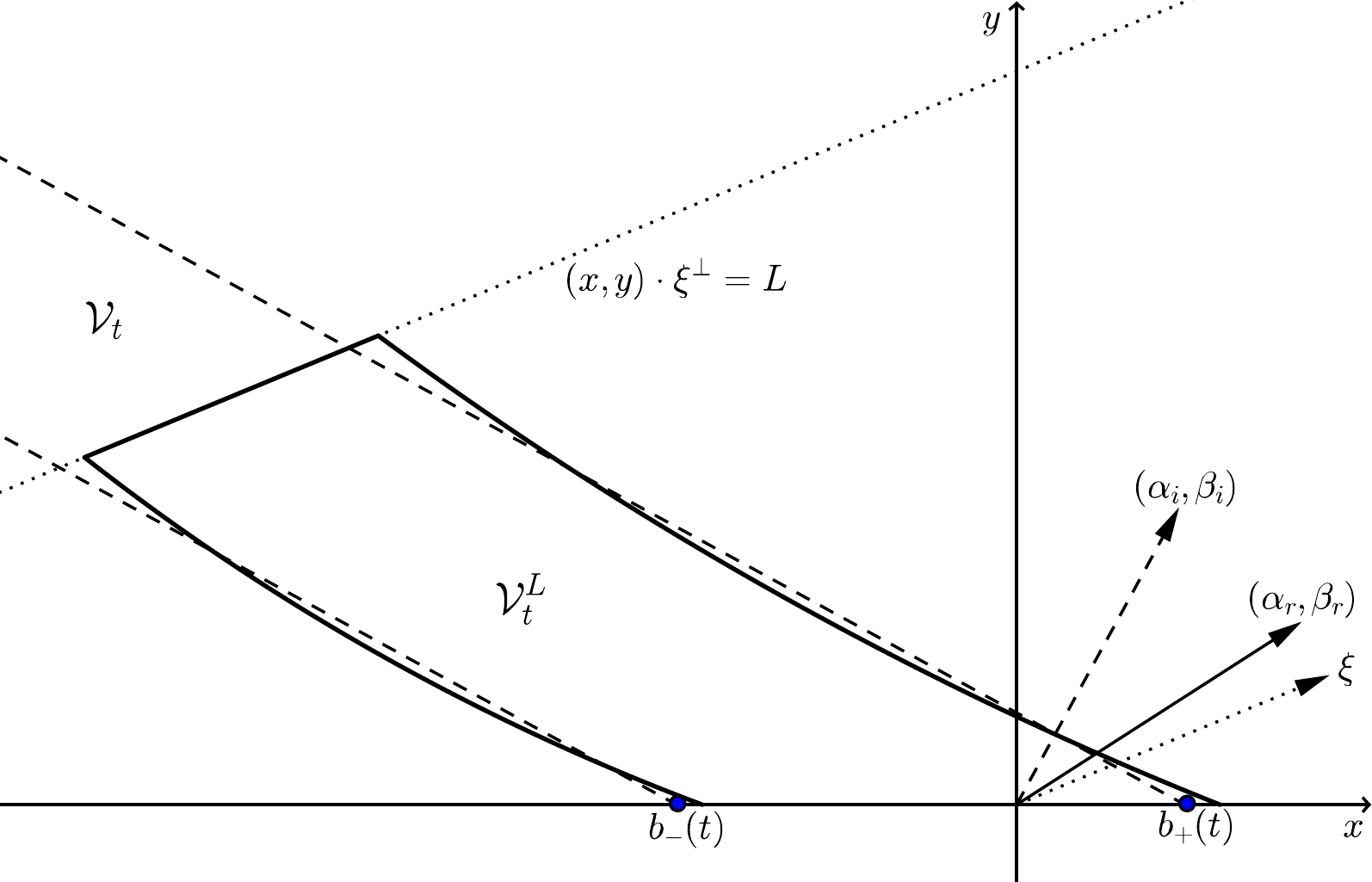}}
 \caption{Relations between the slopes of $\xi$, $(\alpha_r,\beta_r)$ and 
$(\alpha_i,\beta_i)$.}
 \label{fig:slopes}
 \end{figure}

Consider now the penalised system
\Fi{penal}
\begin{cases}
\partial_t u-D \partial_{xx} u=\nu\rest{v}{y=0}-\mu u-\e(u+v) &
x\in\R,\
t>0\\
\partial_t v-d\Delta v=(f'(0)-\e)v & (x,y)\in\O,\ t>0\\
-d\partial_y\rest{v}{y=0}=\mu u-\nu\rest{v}{y=0}
-\e(u+v) & x\in\R,\ t>0.
\end{cases}
\Ff
A small perturbation $\e$ does not affect the qualitative results of  
Section \ref{sec:exp}
\footnote{ the curves $\Sigma,\Gamma$ are replaced by some curves converging
locally uniformly to $\Sigma,\Gamma$ as $\e\to0$, together with their
derivatives}
nor that of Lemma \ref{lem:C-subsolution}. 
Thus, for $\e$ small enough, there exists $w_*^\e$ such that \eqref{penal}
admits exponential solutions in the form \eqref{exp}
with $\alpha,\beta,\gamma\in\R$ for $c\geq w_*^\e$, and
with $\alpha,\beta,\gamma\in\C\backslash\R$ satisfying \eqref{slopes} for
$c<w_*^\e$ close enough to $w_*^\e$. Moreover, $w_*^\e\to w_*$ as $\e\to0$.
We are interested in the complex ones.
Until the end of Section \ref{sec:subsolutions}, $(U,V)$ will denote an
exponential solution of \eqref{penal}, with $\e>0$
sufficiently small, $\alpha,\beta,\gamma\in\C\backslash\R$ satisfying
\eqref{slopes} and $c<w_*^\e$ close to $w_*^\e$.
Changing the sign to the imaginary part of both $U$ and $V$ we still have a
solution. Hence, by \eqref{slopes}, it is not restrictive to assume that
$\Im\alpha,\Im\beta>0$.

We set for short $\alpha_r:=\Re\alpha$, $\alpha_i:=\Im\alpha$,
$\beta_r:=\Re\beta$, $\beta_i:=\Im\beta$. 
 Since $\gamma^{-1}=(\nu+\e+d\beta)/(\mu-\e)$ by the last 
equation of \eqref{penal}, it follows that $\Arg(\gamma^{-1})\in(0,\pi/2)$. 
Resuming, we have:
\Fi{slopes2}
\alpha_r,\alpha_i,\beta_r,\beta_i>0,\qquad
\frac{\alpha_i}{\beta_i}<\frac{\alpha_r}{\beta_r}<\frac{\xi_1}{\xi_2},
\qquad\Arg(\gamma^{-1})\in(0,\pi/2).
\Ff

%%%%%%%%%%%%%%%%%%%%%%%%%%%%%%%%%%%%%%%%%%%%%%%%%%%%%%%%%%%%%%%%%%

\subsubsection{Truncating the exponential solution and the equation in the
field}

The pair $(u,v)$ defined by
$$u:=\Re U=e^{-(\alpha_r,\beta_r)\.[(x,0)-ct\xi]}
\cos((\alpha_i,\beta_i)\.[(x,0)-ct\xi]),$$
$$v:=\Re V=|\gamma|e^{-(\alpha_r,\beta_r)\.[(x,y)-ct\xi]}
\cos((\alpha_i,\beta_i)\.[(x,y)-ct\xi]-\Arg\gamma),$$
is a real solution of \eqref{penal}. 
Consider the following connected components of the positivity sets of
$u$, $v$ at time $0$:
$$\mc{U}=\left(-\frac\pi{2\alpha_i},\frac\pi{2\alpha_i}\right),\qquad
\mc{V}:=\{(x,y)\in\R^2\ :\
(\alpha_i,\beta_i)\.(x,y)\in(-\frac\pi2+\Arg\gamma,\frac\pi2+\Arg\gamma)\}.$$
As the time $t$ increases, these connected components are shifted, becoming
$$\mc{U}_t:=\mc{U}+ct\{\xi_1+\frac{\beta_i}{\alpha_i}\xi_2\},\qquad
\mc{V}_t:=\mc{V}+ct\{\xi\}.$$
In order to truncate the sets $\mc{V}_t$ we consider the reflection with
respect to the line $(x,y)\.\xi^\perp=L$, with $L>0$, where, we recall,
$\xi^\perp:=(-\cos\vt,\sin\vt)$. Namely
$$\mc{R}^L(x,y)=(x,y)+2(L-(x,y)\.\xi^\perp)\xi^\perp.$$
We then define
$$V^L(x,y,t):=V(\mc{R}^L(x,y),t),\qquad v^L:=\Re V^L.$$
The function $v-v^L$ vanishes on $(x,y)\.\xi^\perp=L$ and satisfies the second
equation of~\eqref{penal}. 
The quotient $|V^L|/|V|$ satisfies
$$\frac{|V^L|}{|V|}=\frac{e^{-(\alpha_r,\beta_r)\.[\mc{R}^L(x,y)-ct\xi]}}
{e^{-(\alpha_r,\beta_r)\.[(x,y)-ct\xi]}}=e^{-2(\alpha_r,\beta_r)\.\xi^\perp
(L-(x,y)\.\xi^\perp)}.$$
Let us call $\sigma:=(\alpha_r,\beta_r)\.\xi^\perp$. It follows
from \eqref{slopes2} that $\sigma>0$. Hence,
\Fi{vLbounds}
\frac{|V^L|}{|V|}\leq1\quad\text{ if }\,(x,y)\.\xi^\perp\leq L,\qquad
\frac{|V^L|}{|V|}\leq e^{-\sigma L}\quad\text{ if }\,(x,y)\.\xi^\perp\leq\frac
L2.
\Ff
We deduce that, when restricted to the half-plane $\{(x,y)\.\xi^\perp\leq
L/2\}$, a connected component of the set where 
$(v-v^L)$ is positive at time $t$, denoted by $\mc{V}^L_t$, converges in 
Hausdorff 
distance to $\mc{V}_t$ as $L\to\infty$, uniformly in $t\geq0$.
We can now define
$$\ul u(x,t):=\begin{cases}
               u(x,t) & \text{if }x\in\mc{U}_t\\
               0 & \text{otherwise},
              \end{cases}
\qquad
\ul v(x,y,t):=\begin{cases}
               (v-v^L)(x,y,t) & \text{if }(x,y)\in\mc{V}^L_t\\
               0 & \text{otherwise}.
              \end{cases}$$
We claim that $\ul v$ is bounded. The set $\mc{V}^L_t$ satisfies
$$\mc{V}^L_t\subset\{(x,y)\.\xi^\perp\leq L\}\cap
\{-\pi+\Arg\gamma\leq(\alpha_i,\beta_i)\.[(x,y)-ct\xi]\leq\pi+\Arg\gamma\},$$
as it is seen by noticing that $v=-|V|$ on the boundary of the latter set and
$|v^L|\leq|V|$ if $(x,y)\.\xi^\perp\leq L$.
Thus
$$\ul v\leq2|\gamma|\sup_{\su{(x,y)\.\xi^\perp\leq L}
{[(x,y)-ct\xi]\.(\alpha_i,\beta_i)\geq-\pi+\Arg\gamma}}
e^{-(\alpha_r,\beta_r)\.[(x,y)-ct\xi]}
=2|\gamma|\sup_{\su{(x,y)\.\xi^\perp\leq L}
{(x,y)\.(\alpha_i,\beta_i)\geq-\pi+\Arg\gamma}}
e^{-(\alpha_r,\beta_r)\.(x,y)}.$$
It follows from geometrical considerations that
the latter supremum is finite, c.f.~Figure \ref{fig:slopes} (b). Analytically,
one sees that it is finite if and only if
$$\{(x,y)\.(-\xi^\perp)\geq0\}\cap\{(x,y)\.(\alpha_i,\beta_i)\geq0\}\subset
\{(x,y)\.(\alpha_r,\beta_r)\geq0\},$$
which is equivalent to
% $$[\xi\.(\alpha_r,\beta_r)][\xi\.(\alpha_i,\beta_i)]\geq0,\qquad
% [(\beta_i,\alpha_i)\.(-\xi^\perp)]
% [(\beta_i,\alpha_i)\.(\alpha_r,\beta_r)]\geq0,$$
require that
$(\alpha_r,\beta_r)=\lambda_1(-\xi^\perp)+\lambda_2(\alpha_i,\beta_i)$ with
$\lambda_1,\lambda_2\geq0$.
This property holds true by \eqref{slopes2}.
% that the latter supremum is a maximum, attained at $(x,y)$ such that 
% $(x,y)\.\xi^\perp=L$, $(x,y)\.(\alpha_i,\beta_i)=-\pi+\Arg\gamma$.
We therefore have that $(\ul u,\ul v)$ is bounded. Furthermore,
$\ul v$ is a generalized subsolution of the second equation of~\eqref{penal}. 
Since $f(s)\geq(f'(0)-\e)s$ for $s>0$ small enough, 
we can renormalise $(\ul u,\ul v)$ in such a way that $\kappa\ul v$ is a
generalized subsolution of the second equation of~\eqref{Cauchy} too, for all 
$\kappa\in[0,1]$. 
% Even if it means translating in the $x$ variable, we can assume that 
% $\ul v(0,0,0)>0$.
Next, like $v$, $v^L$ satisfies $v^L((x,y)+ct\xi,t)=v^L(x,y,0)$ and thus
\eqref{sub>0} holds.
It only remains to show that $(\ul u,\ul v)$ is a generalized subsolution of the
equations on the road in the sense of Definition \ref{def:gensub}.

%%%%%%%%%%%%%%%%%%%%%%%%%%%%%%%%%%%%%%%%%%%%%%%%%%%%%%%%%%%%%%%%%%

\subsubsection{The equations on the road}

Let us write
$$\mc{U}_t=(a_-(t),a_+(t)),\qquad
\mc{V}_t\cap\{y=0\}=(b_-(t),b_+(t))\times\{0\}.$$
Since $\Arg(\gamma^{-1})\in(0,\pi/2)$ by \eqref{slopes2}, we deduce that
$$b_-(t)=a_-(t)-\frac{\Arg(\gamma^{-1})}{\alpha_i}<a_-(t)<b_-(t)+
\frac\pi{\alpha_i}=b_+(t)=a_+(t)-\frac{\Arg(\gamma^{-1})}{\alpha_i}<a_+(t).$$
We further see that
\Fi{fase-u}
\frac{u(b_\pm(t)+x,t)}{|U(b_\pm(t)+x,t)|}
=\pm\sin(-\alpha_ix+\Arg(\gamma^{-1})).
\Ff
\Fi{fase-v}
\frac{v(b_\pm(t)+x,0,t)}{|V(b_\pm(t)+x,0,t)|}
=\mp\sin(\alpha_ix).
\Ff
% As a consequence, 
% \Fi{fase}
% \frac{u(b_\pm(t),t)}{|U(b_\pm(t),t)|}
% =\pm\sin(\Arg(\gamma^{-1})),\qquad
% \frac{v(a_\pm(t),0,t)}{|V(a_\pm(t),0,t)|}=\mp\sin(\Arg(\gamma^{-1})).
% \Ff
For $t\geq0$ and $(x,0)\in\mc{V}_t$ we see that
$$x>b_-(t)=\frac1{\alpha_i}(-\frac\pi2+\Arg\gamma+ct(\alpha_i,\beta_i)\.\xi)
\geq\frac1{\alpha_i}(-\frac\pi2+\Arg\gamma),$$
whence
$$(x,0)\.\xi^\perp=-\xi_2 x<\frac{\xi_2}{\alpha_i}(\frac\pi2-\Arg\gamma).$$
It follows that $\mc{V}_t\cap\{y=0\}$ is contained in $\{(x,y)\.\xi^\perp\leq
L/4\}$ for $L$ large enough and $t\geq0$. Thus, the sets
$\mc{V}^L_t\cap\{y=0\}$ approach $(b_-(t),b_+(t))\times\{0\}$ as $L\to+\infty$,
uniformly with respect to $t\geq0$. We consider separately the two equations on
the road. Below, the time $t\geq0$ is fixed and the expressions depending on the
$y$-variable are always understood at $y=0$.

{\em The third equation of \eqref{Cauchy}.}\\
The condition involving the third equation of \eqref{Cauchy} in Definition
\ref{def:gensub} is trivially satisfied if $\ul v=0$. Otherwise, if $\ul v>0$,
then $(x,0)\in\mc{V}^L_t$ and there holds
$$-d\partial_y\ul v+\nu\ul v\leq\mu u-\e(u+v)+h|V^L|,$$
for some $h>0$ only depending on $\alpha,\beta,\xi^\perp$.
For $L$ large enough, $\mc{V}^L_t\cap\{y=0\}$ is contained in
$\{(x,y)\.\xi^\perp\leq L/2\}$ and then \eqref{vLbounds} yields
\Fi{3}
-d\partial_y\ul v+\nu\ul v\leq\mu u-\e(u+v)+h|V|e^{-\sigma L}.
\Ff
By \eqref{fase-u}, there exists $k,\delta_0>0$ only depending on $\alpha_i$ and
$\Arg(\gamma^{-1})$ such that, for $\delta\in(0,\delta_0)$,
\Fi{bounds-u}
\frac{u(x,t)}{|U(x,t)|}<-k\quad\text{if }|x-b_-(t)|<\delta,\qquad
\frac{u(x,t)}{|U(x,t)|}>k\quad\text{if }|x-b_+(t)|<\delta.
\Ff
Our aim is to show that, for $\delta$ small and $L$ large
enough independent of $t$, $(\ul u,\ul v)$ is a generalised subsolution of the
last equation of \eqref{Cauchy} for $x\in[b_-(t)-\delta,b_+(t)+\delta]$. 
Thus, up to increasing $L$ in such a way that
$\mc{V}^L_t\cap\{y=0\}\subset(b_-(t)-\delta,b_+(t)+\delta)\times\{0\}$ for all
$t\geq0$, it is a generalised subsolution of that equation everywhere.

We first focus on a neighbourhood of $b_+$, where $\ul u=u$. From \eqref{3},
using \eqref{fase-v}, \eqref{bounds-u} and recalling that $|V|=|\gamma||U|$, we
obtain, for $|x-b_+(t)|<\delta$,
\[\begin{split}
-d\partial_y\ul v+\nu\ul v-\mu\ul u &\leq-\e(u+v)+h|V|e^{-\sigma L}\\
&<[-\e(k-|\gamma|\alpha_i\delta)+h|\gamma|e^{-\sigma L}]|U|.
\end{split}\]
Choosing then $\delta\leq k/(2|\gamma|\alpha_i)$ yields
$$-d\partial_y\ul v+\nu\ul v-\mu\ul u <\left(-\frac{\e k}2+h|\gamma|e^{-\sigma
L}\right)|U|.$$
We eventually infer that, for $L$ large
enough independent of $t$, $(\ul u,\ul v)$ is a generalised subsolution of the
last equation of \eqref{Cauchy} in the $\delta$ neighbourhood of $b_+(t)$. 
Consider now points such that $|x-b_-(t)|<\delta$, where $\ul u=0$. By
\eqref{3} we get
\[\begin{split}
-d\partial_y\ul v+\nu\ul v-\mu\ul u &\leq(\mu-\e)u-\e v+h|V|e^{-\sigma L}\\
&<[-(\mu-\e)k+\e|\gamma|\alpha_i\delta+h|\gamma|e^{-\sigma L}]|U|,
\end{split}\]
provided that $\e<\mu$. Taking $\e<\mu/2$ we end up with the same inequality as
in the case $|x-b_+(t)|<\delta$ treated above. It remains the case
$x\in[b_-(t)+\delta,b_+(t)-\delta]$.
There we have that $v\geq k'|V|$, for some $k'>0$ only depending on $\alpha_i$,
$\delta$. Consequently, using the fact that $\ul u=\max(u,0)$, we obtain
\[\begin{split}
-d\partial_y\ul v+\nu\ul v &\leq(\mu-\e)u-\e v+h|V|e^{-\sigma L}\\
&\leq\mu\ul u-(\e k'-he^{-\sigma L})|V|.
\end{split}\]
We get again a subsolution for $L$ large enough.

{\em The second equation of \eqref{Cauchy}.}\\
The non-trivial case is $x\in\mc{U}_t=(a_-(t),a_+(t))$, where
$$\partial_t\ul u+\mu\ul u-\nu\ul v=(\nu-\e)v-\e u-\nu\ul v.$$
If $x\in[b_+(t),a_+(t))$ then $\partial_t\ul u+\mu\ul u-\nu\ul v\leq0$,
provided that $\e\leq\nu$.
As before, let $k,\delta_0>0$ be such that \eqref{bounds-u} holds for
$\delta\in(0,\delta_0)$. 
Using \eqref{fase-v} and the equality
$|V|=|\gamma||U|$ we get, if $|x-b_+(t)|<\delta$, 
$$\partial_t\ul u+\mu\ul u-\nu\ul v\leq
[|\nu-\e||\gamma|\alpha_i\delta-\e k]|U|,$$
which is negative for $\delta$ small, independent of $t$.
Consider the remaining case $x\in(a_-(t),b_+(t)-\delta]$.
There, from one hand $v\geq k'|V|$  with $k'$ only
depending on $\alpha_i,\gamma,\delta$, from the other, by \eqref{vLbounds},
$v^L\leq|V|e^{-\sigma L}$ provided that $L$ is large enough in such a way that
$-a_-(t)\cos\vt\leq L/2$. Hence,
$$\partial_t\ul u+\mu\ul u-\nu\ul v=\nu v^L-\e v-\e u
\leq(\nu e^{-\sigma L}-\e k')|V|.$$
We eventually infer that, for $L$ large
enough independent of $t$, $(\ul u,\ul v)$ is a generalised subsolution of the
second equation of \eqref{Cauchy}. This concludes the proof of Lemma
\ref{lem:gensub}.

%%%%%%%%%%%%%%%%%%%%%%%%%%%%%%%%%%%%%%%%%%%%%%%%%%%%%%%%%%%%%%%%%%

\section{Proof of the spreading property}
\label{sec:spreading}

In this section we show that the set $\mc{W}$ defined in Section
\ref{sec:exp} is indeed the asymptotic expansion shape of the system
\eqref{Cauchy}.
This proves \thm{main} part (i).
Moreover, by the definition of the critical angle $\vt_0$, part (iii) also
follows.

We show separately that solutions spread at most and at least with the velocity
set $\mc{W}$, c.f.~\eqref{W+} and \eqref{W-} respectively.
The upper bound \eqref{W+} follows by comparison with the planar waves of 
Section \ref{sec:exp}. The proof of \eqref{W-} is more
involved. It combines the convergence result close to the road given by
\cite{BRR2} with the existence of compactly supported subsolutions provided by
Lemma \ref{lem:gensub}. Then one concludes using a
standard Liouville-type result for strictly positive solutions.

Throughout this section, $(u,v)$ denotes a solution of \eqref{Cauchy} with an
initial datum $(u_0,v_0)\not\equiv(0,0)$ compactly supported. As already 
mentioned in the introduction, the well-posedness of the Cauchy problem is 
proved in \cite{BRR2}.

%%%%%%%%%%%%%%%%%%%%%%%%%%%%%%%%%%%%%%%%%%%%%%%%%%%%%%%%%%%%%%%%%%

\subsection{The upper bound}

\begin{proof}[Proof of \eqref{W+}]
We prove \eqref{W+} showing that, for any $\e>0$, there exists $T>0$
such that the following holds:
$$\forall\vt\in[-\pi/2,\pi/2],\ c\geq w_*(\vt)+\e,\ t\geq T,\quad
v(ct\sin\vt,ct\cos\vt,t)<\e.$$
By symmetry, we can restrict ourselves to $\vt\in[0,\pi/2]$.
Let $R>0$ be such that
$$\supp\, u_0\subset[-R,R],\qquad\supp\, v_0\subset\ol B_R.$$
For $\vt\in[-\pi/2,\pi/2]$, let $(U_\vt,V_\vt)$ be the planar wave for the
linearised system \eqref{linear} defined by \eqref{exp} 
with
$\xi=(\sin\vt,\cos\vt)$, $c=w_*(\vt)$,
$\alpha=\alpha_*(\vt)$, $\beta=\beta_*(\vt)$ and
$\gamma=\mu/(\nu+d\beta_*(\vt))$. 
It is straightforward to check that the functions $\alpha_*$ and
$\beta_*$ are continuous, hence bounded.
Since for $\vt\in[0,\pi/2]$ it holds that
$$\forall(x,y)\in\ol B_R,\quad
U_\vt(x,0)\geq e^{-|R|\alpha_*(\vt)},\quad
V_\vt(x,y,0)\geq\frac{\mu}{\nu+d\beta_*(\vt)}
e^{-|R|(\alpha_*(\vt),\beta_*(\vt))},$$
there exists $\kappa>0$, independent of $\vt$, such that all the
$\kappa(U_\vt,V_\vt)$
are above $(u,v)$ at time $0$. The pairs $\kappa(U_\vt,V_\vt)$ are still
supersolutions
of \eqref{linear}, and then of \eqref{Cauchy} because, by the KPP hypothesis,
$f'(0)\kappa V_\vt\geq f(\kappa V_\vt)$.
The comparison principle then yields that, for $\vt\in[0,\pi/2]$ and $t\geq0$,
$\kappa V_\vt\geq v$, whence, in particular, 
$$\forall c\geq0,\quad
v(ct\sin\vt,ct\cos\vt,t)\leq\frac{\kappa\mu}{\nu+d\beta_*(\vt)}
e^{-(c-w_*(\vt))t(\alpha_*(\vt),\beta_*(\vt))\.(\sin\vt,\cos\vt)}.$$
Notice now that the functions $\alpha_*$ and
$\beta_*$ are strictly positive, excepted at $0$ where $\alpha_*=0$,
$\beta_*\neq0$, and
at $\pi/2$ where $\alpha_*\neq0$, $\beta_*=0$ if $D\leq2d$. 
It follows that $(\alpha_*(\vt),\beta_*(\vt))\.(\sin\vt,\cos\vt)$ is positive
on $[0,\pi/2]$, thus it has a positive minimum by continuity.
The result then follows.
\end{proof}

%%%%%%%%%%%%%%%%%%%%%%%%%%%%%%%%%%%%%%%%%%%%%%%%%%%%%%%%%%%%%%%%%%

\subsection{The lower bound}
\label{sec:lower}

% \begin{lemma}\label{lem:triangle}
% Let $\vt\in(-\pi/2,\pi/2)$, $c_f\in(0,w_*(\vt))$, $c_r\in(0,w_*(\pi/2))$ and
% call
% $$\mc{T}:=\{a c_f(\sin\vt,\cos\vt)+b c_r(1,0)\ :\ 
% a\geq0,\ a+|b|\leq1\}.$$
% Then, 
% $$\inf_{\su{t\geq1}{\frac1t(x,y)\in\mc{T}}}v(x,y,t)>0.$$
% \end{lemma}

\begin{proof}[Proof of \eqref{W-}]
% It is not restrictive to assume that $\e<c_K$.
We first show that $v$ is
bounded from below away from $0$ in some  
suitable expanding sets. This allows us to conclude by means of a standard
Liouville-type result for entire solutions with positive infimum.

Step 1.{ \em  For $\e\in(0,c_K)$ and $\vt\in[-\pi/2,\pi/2]$, there exist
$(\hat x,\hat y)\in\ol\O$ and an open set $A$ in the relative topology of
$\ol\O$ such that
$$A\supset\{r(\sin\vt,\cos\vt)\ :\ 0\leq r\leq w_*(\vt)-\e\},\qquad
\inf_{\su{t\geq1}{(x,y)\in t A}}v(\hat x+x,\hat y+y,t)>0.$$}\\
Consider the case $\vt\neq\pm\pi/2$.
Let $(\ul u,\ul v)$ be a generalised subsolution given by Lemma
\ref{lem:gensub}, with $c>w_*(\vt)-\e>0$, and set 
$$\delta:=\frac{c-w_*(\vt)+\e}{2c}\,\in(0,1/2).$$
% Let $T>0$ be such that $\rest{v}{t=T}>$
Even if it means multiplying $\ul u$, $\ul v$ by a small factor $\kappa>0$, we
can assume that
$\sup\,\rest{\ul u}{t=0}<\nu/\mu$, $\sup\,\rest{\ul v}{t=0}<1$.
We now make use of the spreading result from \cite{BRR2}, summarized 
here by \eqref{e1.5}. Recalling that the $c_*$ there coincides 
with $w_*(\pi/2)$, the second limit implies the
existence of $\tau>0$
such that, for $\lambda\in(\delta,1]$ and $|c'|<w_*(\pi/2)-\e/2$,
the following holds true:
$$\forall(x,y)\in\ol\O,\ t\geq\tau,\quad
v(x+c'\lambda t,y,\lambda t)>\ul v(x,y,0),\quad
u(x+c'\lambda t,\lambda t)>\ul u(x,0).$$
Then, by comparison, $v(x+c'\lambda t,y,\lambda t+s)\geq\ul v(x,y,s)$ for
$t\geq\tau$ and $s\geq0$, from
which, taking $s=(1-\lambda)t$ and $(x,y)=(\hat
x,\hat y)+c(1-\lambda)t(\sin\vt,\cos\vt)$, where $(\hat
x,\hat y)$ is such that \eqref{sub>0} holds, we get
$$v(\hat x+[c(1-\lambda)\sin\vt+c'\lambda]t,\hat y+[c(1-\lambda)\cos\vt]t,t)
>\ul v(\hat x,\hat y,0)>0.
$$
Namely,
$$\inf_{\su{t\geq\tau}{(x,y)\in tA}}v(\hat x+x,\hat y+y,t)>0,$$
where $A$ is the following set: 
$$A=\{(c(1-\lambda)\sin\vt+c'\lambda,c(1-\lambda)\cos\vt)\ :\
\delta<\lambda\leq1,\ |c'|<w_*(\pi/2)-\e/2\},$$
which is open in the relative topology of $\ol\O$.
By the choice of $\delta$, restricting to the values $c'=0$ and
$2\delta\leq\lambda\leq1$ in the expression of $A$ we recover the segment
$\{r(\sin\vt,\cos\vt)\ :\ 0\leq r\leq w_*(\vt)-\e\}$. 
While, restricting to $\lambda=1$ and $|c'|\leq w_*(\pi/2)-\e$, we
obtain $[-w_*(\pi/2)+\e,w_*(\pi/2)-\e]\times\{0\}$,
which is the sought segment in the case $\vt=\pm\pi/2$.
% It is cleat that, up to increasing $\tau$, there exists another open set
% $A_\vt$ still containing the above two segments, and such that $t A_\vt\subset
% t A+\{(\hat x,\hat y)\}$. 
The proof of the step 1 is thereby complete, because the minimum of $v$ 
on compact subsets of $\ol\O\times[1,\tau]$ is positive by the strong comparison
principle with $(0,0)$.

% Step 2.{ \em $\forall\e\in(0,c_K)$, there is a rectangle
% $Q=[-w_*(\pi/2)-\e/2,w_*(\pi/2)+\e/2]\times[0,\delta]$ such that
% $$A\supset\{r(\sin\vt,\cos\vt)\ :\ 0<r\leq w_*(\vt)-\e\},\qquad
% \inf_{\su{t\geq1}{\frac1t(x,y)\in Q}}v(x,y,t)>0.$$}\\
% Let $(\ul u,\ul v)$ be the function given by Lemma \ref{lem:gensub}, associated
% with $\vt=0$ and, say, $c>c_K/2$.
% Let $R>0$ be such that $\supp\rest{\ul v}{t=0}\subset B_R$.
% Up to multiplying $(\ul u,\ul v)$ by $\kappa>0$ small enough, we can assume that
% $\rest{\ul v}{t=0}\leq1/2$.
% Let $\delta\in(0,1)$ to be chosen later.
% We know from \cite{BRR2} \nota{quote} that, for $\tau$ large enough, $v$
% satisfies
% $$\forall t\geq\tau,\quad
% \inf_{\su{|x|\leq w_*(\pi/2)(1-\delta)t}{0\leq y\leq
% R}}v(x,y,t)>\frac12.$$
% For $t\geq\tau$, $|z|\leq w_*(\pi/2)(1-\delta)t-R$ and $(x,y)\in\ol\O$, we find
% that $v(x,y,t)>1/2\geq \ul v(x-z,y,0)$. Whence, by comparison,
% $$\forall t\geq\tau,\ |z|\leq w_*(\pi/2)(1-\delta)t-R,\quad
% v(\hat x,\hat y,t)>\frac12.$$

% Step 2.{ \em For $\e\in(0,c_K)$, there exists $T>0$ such that
% $$\inf_{\su{(x,y)\in\ol\O,\ t\geq T}{\dist(\frac1t(x,y),
% \ol\O\backslash\W)\geq\e}}v(x,y,t)>0.$$}\\
% The family $(A_\vt)_{\vt\in[-\pi/2,\pi/2]}$ is an open cover of the compact
% set $\{(x,y)\in\ol\O\ :\ \dist((x,y),\ol\O\backslash\W)\geq\e\}$.
% Let $(A_{\vt_i})_{i=1,\dots,N}$ be a finite subcover.
% 
% given by the step
% 1. By compactness, there is a subfamily 

Step 2.{ \em Conclusion.}\\
Fix $\e\in(0,c_K)$.
Let $((x_n,y_n))_{n\in\N}$ be a sequence in $\ol\O$ and $\seq{t}$
a sequence in $\R_+$ such that 
$$\limn t_n=+\infty,\qquad \forall n\in\N, \quad
\dist\left(\frac1{t_n}(x_n,y_n),\ol\O\backslash\W\right)>\e.$$
By the boundedness of $v$ it follows that $(v(x_n,y_n,t_n))_{n\in\N}$ converges
up to subsequences. In order to prove \eqref{W-} we need to show that the limits
of all converging subsequences are equal to 1. Let us still call
$(v(x_n,y_n,t_n))_{n\in\N}$ one of such subsequences and set
$$m:=\limn v(x_n,y_n,t_n).$$
If $\seq{y}$ admits a bounded subsequence $(y_{n_k})_{k\in\N}$ then,
since 
$$\e<\dist\left(\frac1{t_{n_k}}(x_{n_k},y_{n_k}),\ol\O\backslash\W\right)\leq
\dist\left(\frac{x_{n_k}}{t_{n_k}},\R\backslash[-w_*(\pi/2),w_*(\pi/2)]\right)
+\frac{y_{n_k}}{t_{n_k}},$$
we derive
$|x_{n_k}|\leq(w_*(\pi/2)-\e/2)t_{n_k}$ for $k$ large enough.
It then follows from \eqref{e1.5} that $m=1$ in this case. 
Consider now the case where $\seq{y}$ diverges.
Let us write $1/t_n(x_n,y_n)=r_n(\sin\vt_n,\cos\vt_n)$, with
$|\vt_n|\leq\pi/2$ and $0\leq r_n\leq w_*(\vt_n)-\e$, and call $\vt$, $r$ the
limit of (a subsequence of) $\seq{\vt}$, $\seq{r}$ respectively.
The continuity of $w_*$ yields $0\leq r\leq w_*(\vt)-\e$.
% Notice that it may happen that $r=0$ or that $\vt=\pm\pi/2$.
% but, in any case, $\limn
% y_n=+\infty$,
% because $v$ converges to $1$ as $t\to\infty$, uniformly in $|x|\leq
% (w_*(\vt)-\e/2)t$ and $y$ bounded.
Consider the sequence of functions $\seq{v}$ defined by
$$v_n(x,y,t):=v(x+x_n,y+y_n,t+t_n).$$
For $n$ large enough, the $v_n$ are defined in any given
$K\subset\subset\R^2\times\R$ and, by interior parabolic estimates (see, e.g.,
\cite{Lady}) 
they are uniformly bounded in $C^{2,\delta}(K)$ and $C^{1,\delta}(K)$
with respect to the space and time variables respectively, for some
$\delta\in(0,1)$. Hence, $\seq{v}$
converges (up to subsequences) locally uniformly to 
a solution $v_\infty$ of 
\Fi{vinfty}
\partial_t v_\infty-d\Delta v_\infty=f(v_\infty),\quad (x,y)\in\R^2,\ t\in\R.
\Ff
Moreover, $v_\infty(0,0,0)=m$.
Consider the point $(\hat x,\hat y)$ and the set $A$ given by the step 1,
associated with $\e$ and $\vt$. For $(x,y)\in\R^2$ and $t\in\R$, we see that
$$\limn\frac1{t+t_n}(x+x_n-\hat x,y+y_n-\hat y)=r(\sin\vt,\cos\vt)\in A.$$
Thus, for $n$ large enough, since $y+y_n-\hat y>0$ and 
$A$ is open in $\ol\O$, we have that $(x+x_n-\hat x,y+y_n-\hat y)\in
(t+t_n)A$, whence $v_n(x,y,t)\geq h>0$, with $h$
independent of $(x,y,t)$. It follows that $v_\infty\geq h$ in all
$\R^2\times\R$. 
Since $f>0$ in $(0,1)$ and $f<0$ in $(1,+\infty)$, it is straightforward to
see by comparison with solutions of the ODE $z'=f(z)$ in $\R$, that the unique
bounded solution of \eqref{vinfty} which is bounded from below away from $0$ is
$v_\infty\equiv1$. As a consequence,
% 
% Let $z_\pm:[0,+\infty)\to\R$ be the solutions of
% $z_\pm'=f(z_\pm)$ with
% $z_-(0)=h$ and $z_+(0)=\sup v_\infty$ respectively. Clearly, $z_\pm(+\infty)=1$.
% For $T>0$ we have that $z_-(0)\leq\rest{v}{t=-T}\leq z_+(0)$, and then, by
% comparison,
% $$\forall (x,y)\in\R^2,\ t\geq0,\quad z_-(t)\leq v_\infty(x,y,t-T)\leq z_+(t).$$
% Taking $(x,y)=(0,0)$, $t=T$ and then letting $T$ go to $+\infty$ in the above
% inequalities we eventually derive 
$m=v_\infty(0,0,0)=1$, which concludes the
proof of \eqref{W-}.
\end{proof}

%%%%%%%%%%%%%%%%%%%%%%%%%%%%%%%%%%%%%%%%%%%%%%%%%%%%%%%%%%%%%%%%%%

\section{Further properties of the function $w_*$}
\label{sec:w*}

We now study the function $w_*:[-\pi/2,\pi/2]\to\R_+$ defined in Section
\ref{sec:exp}. This will complete the proof of \thm{main} part (ii).
Since $w_*$ is even, we restrict ourselves to $[0,\pi/2]$. 
If $D\leq2d$ then $w_*\equiv c_K$. Thus, throughout this section, we assume that
$D>2d$. We recall that $(\beta_*(\vt),\alpha_*(\vt))$ is the unique
intersection point between the sets $\mc{S}(w_*(\vt))$ and $\mc{G}(w_*(\vt))$
associated with $\xi=(\sin\vt,\cos\vt)$.

We start with the following observation.
\begin{lemma}\label{lem:tangent}
The function $w_*$ satisfies
$$\forall \vt\in[\vt_0,\pi/2],\ \t\vt\in[0,\pi/2],\quad
w_*(\t\vt)\leq\frac{\cos(\vt-\vp_*(\vt))}{\cos(\t\vt-\vp_*(\vt))}\,w_*(\vt),$$
where $\vp_*(\vt)=\arctan\alpha_*(\vt)/\beta_*(\vt)$.
\end{lemma}

\begin{proof}
Take $\vt,\t\vt$ as in the statement of the lemma.
The pair $(U,V)$ defined by \eqref{exp}, with 
$\xi=(\sin\vt,\cos\vt)$,
$c=w_*(\vt)$, $\alpha=\alpha_*(\vt)$, $\beta=\beta_*(\vt)$ and
$\gamma=\mu/(\nu+d\beta_*(\vt))$, is a solution of \eqref{linear}. We call
$$\t\xi:=(\sin\t\vt,\cos\t\vt),\qquad
\t
c:=\frac{(\alpha_*(\vt),\beta_*(\vt))\.\xi}{(\alpha_*(\vt),\beta_*(\vt))\.\t\xi}
\,w_*(\vt),$$
and we rewrite $(U,V)$ in the following way:
$$(U(t,x),V(t,x,y))=(e^{-(\alpha_*(\vt),\beta_*(\vt))\.((x,0)-\t
ct\t\xi)},\gamma
e^{-(\alpha_*(\vt),\beta_*(\vt))\.((x,y)-\t ct\t\xi)}).$$
Thus, by the definition of $w_*(\t\vt)$, we derive
\begin{equation}\label{eq:line}
w_*(\t\vt)\leq\t c=
\frac{(\alpha_*(\vt),\beta_*(\vt))\.\xi}{(\alpha_*(\vt),\beta_*(\vt))\.\t\xi}
\,w_*(\vt).
\end{equation}
The result then follows.
\end{proof}

\begin{proposition}\label{pro:w*}
The function $w_*$ satisfies
$$w_*\in C^1([0,\pi/2]),\qquad
w_*=c_K\ \text{ in }[0,\vt_0],\qquad 
w_*'>0\ \text{ in }(\vt_0,\pi/2].$$
\end{proposition}

\begin{proof}
The fact that $w_*=c_K$ in $[0,\vt_0]$ is just what defines $\vt_0$,
see Section \ref{sec:exp}.
The smoothness of $w_*$ outside the point $\vt_0$ is an easy consequence
of the implicit function theorem. 
% We use the same notation as in the previous section, but we
% emphasize the dependence on $\vt=\arctan(\xi_1/\xi_2)$ of the various terms.
% The angle $\vt_0$ is characterised by the following
% property:
% $$C(c_K,\vt_0)\in\Sigma(c_K,\vt_0).$$ 
% Since $C$ and $\alpha_D^+$ are of class $C^1$ with respect to $c$ and $\vt$, 
% and $r'(c_K)=+\infty$, it follows that $w_*$ is differentiable at $\vt_0$,
% with $w_*(\vt_0)=0$.
Lemma \ref{lem:tangent} implies that, for fixed $\vt\in(\vt_0,\pi/2)$, 
the smooth function
$\t\vt\mapsto\frac{\cos(\vt-\vp_*(\vt))}{\cos(\t\vt-\vp_*(\vt))}\,w_*(\vt)$
touches $w_*$ from above at the point $\vt$, whence we derive
$$\forall\vt\in(\vt_0,\pi/2),\quad w_*'(\vt)=\tan(\vt-\vp_*(\vt)) w_*(\vt).$$
In particular, $w_*'(\pi/2)=w_*(\pi/2)\beta_*(\pi/2)/\alpha_*(\pi/2)>0$.
For $\vt\in(\vt_0,\pi/2)$, we deduce that $w_*'(\vt)>0$ if and only if 
$\vt>\vp_*(\vt)$,
which is equivalent to $\tan\vt>\alpha_*(\vt)/\beta_*(\vt)$. 
Calling as usual $\xi:=(\sin\vt,\cos\vt)$, this inequality reads
$\xi_1/\xi_2>\alpha_*(\vt)/\beta_*(\vt)$, which holds true by geometrical
considerations, as already seen in the proof of Lemma \ref{lem:C-subsolution},
see Figure \ref{fig:slopes} (a).
%
% We show with geometrical arguments that $\vt>\vp_*(\vt)$ for
% $\vt\in(\vt_0,\pi/2)$, which implies $w_*'(\vt)>0$.
% The centre of $\Gamma((w_*(\vt))$ is 
% $$C(w_*(\vt))=\frac{w_*(\vt)}{2d}(\cos(\vt),\sin(\vt)).$$
% Hence, $\vt$ coincides with the angle formed by the vector $C(w_*(\vt))$ with
% the $\beta$ axis. On the other hand, $\vp_*(\vt)$ is the angle formed by the
% vector $(\beta_*(\vt),\alpha_*(\vt))$ with the $\beta$ axis (see figure
% \ref{fig:monotone}). The inequality $\vt>\vp_*(\vt)$ then follows. 
%
As $\vt\to\vt_0^+$, the disc $\mc{G}(w_*(\vt))$ collapses to the point
$c_K/2d(\cos\vt_0,\sin\vt_0)$, whence $w_*(\vt)\to c_K$,
$\vp_*(\vt)\to\vt_0$ and
eventually
$w_*'(\vt)\to0$.
% 
% We deduce that 
% $$\lim_{\vt\to\vt_0^+}\vp_*(\vt)=\vp(\vt_0),\qquad
% \lim_{\vt\to\vt_0^+}w_*'(\vt)=0.$$
This shows that $w_*'$ is continuous at $\vt_0$ too.
\end{proof}
% 
% \begin{figure}[ht]
% \psfrag{C}{\footnotesize$C(w_*(\vt))$}
% \psfrag{bS}{\footnotesize$\Sigma(w_*(\vt))$}
% \psfrag{t}{\scriptsize$\vt$}\psfrag{f}{\scriptsize$\vp_*(\vt)$}
% \psfrag{a}{\footnotesize$\alpha$}\psfrag{b}{\footnotesize$\beta$}
% \psfrag{(b,a)}{\scriptsize$(\beta_*(\vt),\alpha_*(\vt))$}
% \begin{center}
% \includegraphics[height=7cm]{monotone.eps}
% \caption{The angles $\vt$ and $\vp_*(\vt)$.}
% \label{fig:monotone}
% \end{center}
% \end{figure}
% 

To conclude the proof of \thm{main} part (ii) it remains to show that 
$\mc{W}$ is convex and that 
$$\W\supsetneq\ul{\W}:=\conv\big((\ol B_{c_K}\cap\ol\O)\cup
[-c_*,c_*]\times\{0\}\big),$$
where, we recall, $c_*=w_*(\pi/2)$.
Proposition \ref{pro:w*} implies that $\partial\mc{W}$ is of class $C^1$, except
at the extremal points $(\pm c_*,0)$. The exterior unit normal to $\mc{W}$ at 
those points is understood as the limit of the normals to points of 
$\O\cap\partial\W$ converging to $(\pm c_*,0)$.

\begin{proposition}\label{pro:W}
The set $\mc{W}$ is strictly convex and, 
% for $\vt\in[\vt_0,\pi/2]$, its exterior unit normal 
% $\mathbf{n}(\vt)$ at the point $w_*(\vt)(\sin\vt,\cos\vt)$ is given by
% $$\mathbf{n}(\vt)=(\alpha_*(\vt),\beta_*(\vt)).$$
for $\vt\in(\vt_0,\pi/2]$, its 
exterior unit normal at the point $w_*(\vt)(\sin\vt,\cos\vt)$ is parallel to 
$(\alpha_*(\vt),\beta_*(\vt))$. 

In particular, $\mc{W}\supsetneq\ul{\W}$.
\end{proposition}

\begin{proof}
Fix $\vt\in[\vt_0,\pi/2]$. For
$(x,y)\in\mc{W}\cap\{x\geq0\}$, we write $(x,y)=r(\sin\t\vt,\cos\t\vt)$
for some $\t\vt\in[0,\pi/2]$ and $0\leq r\leq w_*(\t\vt)$.
Using the inequality given by Lemma \ref{lem:tangent} in the form
\eqref{eq:line}, with $\xi=(\sin\vt,\cos\vt)$ and
$\t\xi:=(\sin\t\vt,\cos\t\vt)$,
yields
\[\begin{split}
(\alpha_*(\vt),\beta_*(\vt))\.(x,y)&=r(\alpha_*(\vt),\beta_*(\vt))\.\t\xi\\
&\leq w_*(\t\vt)(\alpha_*(\vt),\beta_*(\vt))\.\t\xi\leq
w_*(\vt)(\alpha_*(\vt),\beta_*(\vt))\.\xi,
\end{split}\]
and equality holds if and only if $(x,y)=w_*(\vt)\xi$. 
This shows that $\W\cap\{x\geq0\}$ is contained in the half-plane
$\{(\alpha_*(\vt),\beta_*(\vt))\.(x,y)<
w_*(\vt)(\alpha_*(\vt),\beta_*(\vt))\.(\sin\vt,\cos\vt)\}$, except for the point
$w_*(\vt)(\sin\vt,\cos\vt)$ which belongs to its boundary. Then, clearly, the 
same property holds for the whole $\mc{W}$. This shows the convexity of $\W$ and 
the directions of the normal vectors.
%\end{proof}
%
%\begin{proposition}\label{pro:W>}
%It holds that $\mc{W}\supsetneq\ul{\W}$.
%\end{proposition}
%
%\begin{proof}

Let us prove the last statement of the proposition.
Proposition \ref{pro:w*} implies that $\W$ contains $\ol
B_{c_K}\cap\ol\O$, whence, being convex, it contains $\ul{\W}$. We prove
that $\W\not\equiv\ul{\W}$ by showing that the (acute) angle $\vp_*$ formed by
$\W$ with the $x$-axis is strictly larger than the one formed by $\ul{\W}$,
which is $\vt_1:=\arcsin(c_K/c_*)$. 
We know from the first part of the proposition that 
$\vp_*=\arctan(\alpha_*/\beta_*)$, where, for short,
$\alpha_*:=\alpha_*(\pi/2)$ and $\beta_*:=\beta_*(\pi/2)$. 
Recall that $(\beta_*,\alpha_*)$ is the tangent point between the sets 
$\mc{S}(c_*)$ and $\mc{G}(c_*)$
associated with $\xi=(1,0)$, defined in Section \ref{sec:exp}.
It then follows from geometrical considerations that $\vp_*>\vt_1$, see
Figure \ref{fig:phi_star}.
\end{proof}
\begin{figure}[ht]
\begin{center}
\includegraphics[height=7.5cm]{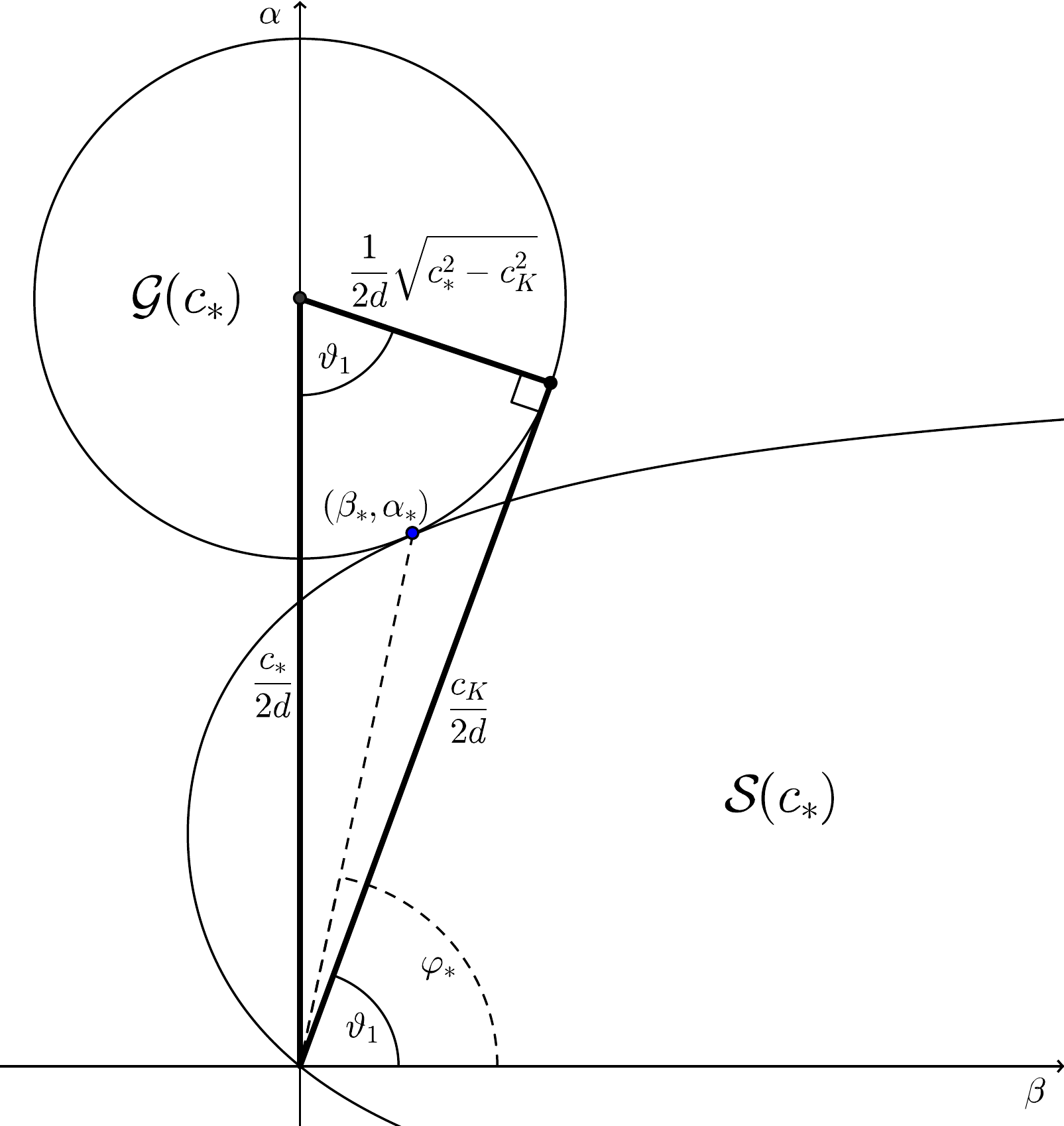}
\caption{The angles $\vp_*$ and $\vt_1$.}
\label{fig:phi_star}
\end{center}
\end{figure}

We deduce from Proposition \ref{pro:W} and Figure \ref{fig:w*>cK} (b) that, for 
$\vt\in(\vt_0,\pi/2]$, the exterior normal at the point 
$w_*(\vt)(\sin\vt,\cos\vt)$ is steeper than $(\sin\vt,\cos\vt)$.

%%%%%%%%%%%%%%%%%%%%%%%%%%%%%%%%%%%%%%%%%%%%%%%%%%%%%%%%%%%%%%%%%%

% 
% 
% \section{Spreading in vertical direction}
% 
% \subsection{Lower bound}
% $$c_*(e_2)\geq c_{KPP}.$$
% 
% For $c<c_{KPP}$ there exists $R>0$ and $\underline v$ subsolution of
% $-c\partial_y v-d\Delta v=f(v)$ in $B_R$ vanishing on $\partial B_R$.
% Moreover,
% up to waiting until $t=1$, $\underline v(x,y+2R)$ is below $v$.
% Let $(\phi,\psi)$ be the sol of \eqref{Cauchy} with initial datum
% $(0,\underline
% v)$. Then $(u,v)>(\phi,\psi)$ and
% $(\tilde\phi,\tilde\psi):=(\phi,\psi)(x,y+ct,t)$
% is increasing in time and then $\tilde\psi$ converges to a nonnegative
% solution
% $V$ of $c\partial_y v-d\Delta v=f(v)$ in $\R^2$. Sliding method $\Rightarrow
% \inf V>\max\underline v$. By usual arguments, $f(\inf V)=0$, that is,
% $V\equiv1$.
% 
% 
% \subsection{Upper bound}
% $$c_*(e_2)\leq c_{KPP}.$$
% 
% Take $c>c_{KPP}$ and call
% $$\lambda:=\frac{c-\sqrt{c^2-c_{KPP}^2}}{2d}.$$
% The couple $(\overline u,\overline v)$ defined by
% $$\overline u(x,t):=\frac1\mu e^{\lambda ct},\qquad
% \overline v(x,y,t):=e^{-\lambda(y-ct)},$$
% is a supersolution to \eqref{Cauchy}. Moreover, up to waiting until some time
% $t$,
% it is above $(u_0,v_0)$. The comparison principle implies that 
% $$\forall \t c>c,\quad (u,v)(x,y+\t ct,t)\to(0,0),$$
% uniformly in $x$ and locally uniformly in $y$.

%%%%%%%%%%%%%%%%%%%%%%%%%%%%%%%%%%%%%%%%%%%%%%%%%%%%%%%%%%%%%%%%%%
Let us finally estimate by how much $\W$ is larger than $\underline\W$.
\begin{proposition}\label{pro:normalspeed}
The family of sets $(t\W)_{t>0}$ evolves with normal speed $c_K$ in the sector 
$\{(\sin\vt,\cos\vt)\ : |\vt|\leq\vt_0\}$ and with normal speed strictly larger than $c_K$ in the sectors $\{(\sin\vt,\cos\vt)\ : \vt_0<|\vt|\leq\pi/2\}$.
\end{proposition}

\begin{proof}
The assertion for the sector $\{(\sin\vt,\cos\vt)\ : |\vt|\leq\vt_0\}$ trivially
holds because $\W$ coincides with $B_{c_K}$ there.
Consider $\vt\in(\vt_0,\pi/2]$ and set $\xi:=(\sin\vt,\cos\vt)$. 
By Proposition \ref{pro:W}, the exterior unit normal to $\W$ at the point 
$w_*(\vt)\xi$ is 
$$\mathbf{n}(\vt):=\frac{(\alpha_*(\vt),\beta_*(\vt))}
{|(\alpha_*(\vt),\beta_*(\vt))|}.$$ Hence, the
speed of expansion of the set $t\W$ at the point $tw_*(\vt)\xi$ in the normal direction $\mathbf{n}(\vt)$ is 
$c_{\mathbf{n}}(\vt):=w_*(\vt)\xi\.\mathbf{n}(\vt).$
This is precisely the normal speed of the level lines of the function $V$ defined by \eqref{exp} with 
$c=w_*(\vt)$, $\alpha=\alpha_*(\vt)$, $\beta=\beta_*(\vt)$ and
$\gamma=\mu/(\nu+d\beta_*(\vt))$. Indeed, we can rewrite
$$V(x,y,t)=\gamma e^{-|(\alpha_*(\vt),\beta_*(\vt))|
[(x,y)\.\mathbf{n}(\vt)-c_{\mathbf{n}}(\vt)t]}.$$
Plugging the above expression in the second equation of 
\eqref{linear} satisfied by $V$, we get
$$c_{\mathbf{n}}(\vt)=\frac{f'(0)}{|(\alpha_*(\vt),\beta_*(\vt))|}
+d|(\alpha_*(\vt), \beta_*(\vt))|.$$
The function $\R_+\ni\lambda\mapsto f'(0)/\lambda+d\lambda$ attains its minimum 
$c_K$ at the unique value $\lambda=\sqrt{f'(0)/d}$. Thus, to prove the 
proposition we need to show that $|(\alpha_*(\vt),\beta_*(\vt))|\neq 
\sqrt{f'(0)/d}$. This follows from the geometrical interpretation of the point 
$P_*\equiv(\beta_*(\vt),\alpha_*(\vt))$, see Figure \ref{fig:w*>cK} (b): the convexity of 
$\mc{S}(c)$ implies that
the angle between the segments $\ol{P_*\,C(c)}$ and $\ol{P_*\,O}$, $O$ denoting 
the origin, is 
larger than $\pi/2$, whence, since these segments have length $\sqrt{c^2-c_K^2}/2d$ 
and $c/2d$ respectively,
elementary considerations about the triangle $O\,P_*\,C(c)$ show that 
$|(\alpha_*(\vt),\beta_*(\vt))|<c_K^2/2d=\sqrt{f'(0)/d}$.
\end{proof}
\section{The case with transport and mortality on the road}
\label{sec:T}

We now describe how to modify the arguments used for problem \eqref{Cauchy} in
order to treat the case of \eqref{CauchyT}. This is
done section by section, keeping the same notation.

{\em Section \ref{sec:exp}.}\\
We need to consider the values $\xi_1\leq0$ too.
The transport and mortality terms affect \eqref{algebraic} through the
additional term $-q\alpha+\rho$ in the left-hand side of the first equation.
This results in the new functions
$$\alpha_D^\pm(c,\beta)=\frac1{2D}\left(c\xi_1-q\pm\sqrt{(c\xi_1-q)^2+
4D(c\xi_2\beta+\chi(d\beta)+\rho)}\right).$$
One can readily check that $\alpha_D^+(c,\beta)$ is still increasing in
$c$ and concave in $\beta$. It further satisfies
the following property, that will be crucial in the sequel:
$\alpha_D^+(c,0)\geq0$.
We can therefore define $w_*$ as before. We have that $w_*=c_K$ if and only if
$C(c_K)\in\mc{S}(c_K)$, which now reads
$$\frac{c_K^2}{2d}-\frac{Dc_K^2}{4d^2}\xi_1^2-\frac{q c_K}{2d}\xi_1+\rho\geq
-\frac{\mu c_K\xi_2}{2\nu+c_K\xi_2}.$$
This inequality can be rewritten in terms of $\xi_1$ as $\Phi(\xi_1)\geq0$,
with 
$$\Phi(s):=2-\frac Dds^2-\frac{2q}{c_K}s+\frac{4d\rho}{c_K^2}
+\frac{4d\mu\sqrt{1-s^2}}{2\nu c_K+c_K^2\sqrt{1-s^2}}.$$
Explicit computation shows that all the above terms are concave in $s$.
Hence, since $\Phi(0)>0$ and $\Phi(\pm\infty)=-\infty$, there are two values
$s_-<0<s_+$ such that $w_*=c_K$ if and only if $\xi_1\in[s_-,s_+]$.
We have that $|s_\pm|<1$ if and only if $\Phi(\pm1)<0$, which is precisely 
condition
\eqref{enhancement}.
Therefore, writing $w_*$ as a function of the angle $\vt$, we derive the
condition for the enhancement of the speed stated in \thm{mainT}, with
$\vt_\pm=\arcsin s_\pm$ if \eqref{enhancement} holds, $\vt_\pm=\pm\pi/2$
otherwise. For $\vt=\pm\pi/2$, we recover the asymptotic
speeds of spreading $c_*^\pm$ in the directions $\pm(1,0)$ given by
Theorem 1.1 of \cite{BRR3}.

{\em Section \ref{sec:subsolutions}.}\\
The only point one has to check is the argument to derive \eqref{slopes}
in the proof of Lemma \ref{lem:C-subsolution}. That argument is
based on the fact that the slope of the tangent line to $\mc{G}(w_*)$ at the
point $(\beta_*,\alpha_*)$ is less than $\alpha_*/\beta_*$, which, in turn, 
is less than $\xi_1/\xi_2$. This properties follow exactly as before, from the
fact that $\alpha_D^+$ is concave in $\beta$ and it is nonnegative at $\beta=0$.

{\em Section \ref{sec:spreading}.}\\
The proof of the upper bound \eqref{W+} works exactly as for \thm{main}.
In the lower bound \eqref{W-}, the value $1$ is now replaced by the function
$V_S(y)$. However, since $V_s(+\infty)=1$, we can proceed exactly as in Section
\ref{sec:lower}, by use of the compactly supported subsolutions and the
convergence result close to the road. The latter is now provided by Theorem
1.1 of \cite{BRR3}.

The arguments in Section \ref{sec:w*} are unaffected by the presence
of the additional terms.

%%%%%%%%%%%%%%%%%%%%%%%%%%%%%%%%%%%%%%%%%%%%%%%%%%%%%%%%%%%%%%%%%%

\begin{appendix}

\section*{Appendix: the generalised comparison principle}
\label{sec:gencom}

\begin{proof}[Proof of Proposition \ref{pro:gencom}]
Following the arguments of the proof of Proposition 3.2 in \cite{BRR2}, 
we start with reducing $(\ol u,\ol v)$ to a strict
supersolution $(\hat u,\hat v)$ which is strictly above $(\ul u,\ul v)$ at time
$0$ and satisfies
\Fi{toinfty}
\lim_{|x|\to\infty}\hat u(x,t)=+\infty,\quad 
\lim_{|(x,y)|\to\infty}\hat v(x,y,t)=+\infty,\quad\text{uniformly
w.r.t.~}t\geq0.
\Ff
% 
% it is
% sufficient to prove the result for the problem
% \begin{equation}\label{Emonotone}
% \begin{cases}
% \partial_t u-D \partial_{xx} u+(\mu+l)u=\nu\rest{v}{y=0}, & x\in\R,\
% t>0 \\
% \partial_t v-d\Delta v=h(t,v), & (x,y)\in\O,\ t>0\\
% -d\partial_y\rest{v}{y=0}+\nu\rest{v}{y=0}=\mu u, & x\in\R,\ t>0,\\
% \end{cases}
% \end{equation}
% where $l$ is the Lipschitz constant of $f$ and $h(t,v):=e^{-lt}f(ve^{lt})-lv$,
% under the additional assumptions that $(\ol u,\ol v)$ is strictly above $(\ul
% u,\ul v)$ at time $0$ and satisfies
% \Fi{toinfty}
% \lim_{|x|\to\infty}\ol u(x,t)=+\infty,\quad 
% \lim_{|(x,y)|\to\infty}\ol v(x,y,t)=+\infty,\quad\text{uniformly
% w.r.t.~}t\geq0.
% \Ff
% 
% \bigskip
To do this, we first multiply $(\ul u,\ul v)$ and $(\ol u,\ol v)$ by
$e^{-lt}$, where $l$ is the Lipschitz constant of $f$, and we end up 
with generalised sub and supersolutions \footnote{formally, but it is
straightforward to verify it in the generalised
sense of Definition \ref{def:gensub}} (still denoted $(\ul u,\ul v)$ and
$(\ol u,\ol v)$)
of the new system
\Fi{Emonotone}
\begin{cases}
\partial_t u-D \partial_{xx} u+(\mu+l)u=\nu\rest{v}{y=0}, & x\in\R,\
t>0 \\
\partial_t v-d\Delta v=h(t,v), & (x,y)\in\O,\ t>0\\
-d\partial_y\rest{v}{y=0}+\nu\rest{v}{y=0}=\mu u, & x\in\R,\ t>0,\\
\end{cases}
\Ff
with $h(t,v):=e^{-lt}f(ve^{lt})-lv$.
In such a way we gain the nonincreasing monotonicity in $v$ of the nonlinear
term $h$. Next, we introduce a nonnegative smooth function $\chi:\R\to\R$ 
satisfying
$$\chi=0\ \text{ in }[0,1],\qquad \lim_{r\to+\infty}\chi(r)=+\infty,
\qquad|\chi''|\leq\delta,$$
where $\delta>0$ will be chosen later.
% $$((N-1)(d+D)+|q|+2|r|)|\chi'|+(2d+D)|\chi''|\leq1
% \ \text{ in }\R.$$
% $$\leq[2\max(|p|,2|q|)]^{-1},\quad
% |\chi''|\leq[2\max(ND,(N+1)d)]^{-1}\quad\text{in }\R.$$
Then, for $\e>0$, we set
$$\hat u(x,t):=\ol u(x,t)+\e(\chi(|x|)+t+1),\qquad
\hat v(x,y,t):=\ol v(x,y,t)+\frac\mu\nu\e(\chi(|x|)+\chi(y)+t+1),$$
% Using the fact that 
% $$|\nabla_x(x\mapsto\chi(|x|))|\leq|\chi'(|x|)|,\qquad
% |\Delta_x(x\mapsto\chi(|x|))|\leq |\chi''(|x|)|+
% (N-1)|\chi'(|x|)|,$$
We claim that $\delta$ can be chosen small enough, independently of $\e$, in
such a way that $(\hat u,\hat v)$ is still a generalised supersolution of
\eqref{Emonotone}, in the strict sense for the first two equations.
Take $\ol x\in\R$ and $\ol t>0$. By the definition of generalised supersolution,
there exists a function $u$ satisfying
$u\geq\ol u$ in a neighbourhood of $(\ol x,\ol t)$ and, at $(\ol x,\ol t)$,
$$u=\ol u,\qquad\partial_t u-D \partial_{xx} u+(\mu+l)u\geq\nu 
\rest{\ol v}{y=0}.$$
The function $\t u(x,t):=u(x,t)+\e(\chi(|x|)+t+1)$ satisfies
$\t u\geq\hat u$ in a neighbourhood of $(\ol x,\ol t)$ and, at $(\ol x,\ol t)$,
$$\t u=\hat u,\qquad\partial_t\t u-D \partial_{xx}\t u+(\mu+l)\t u\geq\nu 
\rest{\hat v}{y=0}+\e(1-D\chi''(|x|)).$$
Then the desired strict inequality holds provided $\delta<1/D$.
For the second equation, we start from a ``test function'' $v$ 
at some $(\ol x,\ol y)\in\O$, $\ol t>0$ and we see that
$\t v(x,y,t):=v(x,y,t)+\frac\mu\nu\e(\chi(|x|)+\chi(y)+t+1)$
satisfies, at $(\ol x,\ol y,\ol t)$,
$$\partial_t \t v-d\Delta\t v\geq h(\ol t,v)+\frac\mu\nu\e(1-2d\delta).$$
If $\delta<1/2d$, the right hand side is strictly larger than $h(\ol t,v)$,
which,
in turn, is larger than $h(t,\t v)$ by the monotonicity of $h$.
The case of the third equation is straightforward.
The claim is thereby proved.

The pair $(\hat u,\hat v)$ is strictly above $(\ul u,\ul v)$ at $t=0$.
Assume by contradiction that $(\hat u,\hat v)$ is not strictly above $(\ul
u,\ul v)$ for all time and call
$$T:=\sup\{t\geq0\ :\ \ul u<\hat u\text{ in }\R\times[0,t],\
\ul v<\hat v\text{ in }\ol\O\times[0,t]\}\,\in[0,+\infty).$$
It follows that $\ul u\leq\hat u$ in
$\R\times[0,T]$, $\ul v\leq\hat v$ in $\ol\O\times[0,T]$. Moreover, 
by \eqref{toinfty} and the continuity of the functions we see
that $T>0$ and either $\hat u-\ul u$ or $\hat v-\ul v$ vanish somewhere at time
$T$. Suppose that $(\hat u-\ul u)(x,T)=0$ for some $x\in\R$. We now use the
fact that $(\ul u,\ul v)$ and $(\hat u,\hat v)$ are a subsolution
and a strict supersolution respectively of \eqref{Emonotone}, in the generalised
sense.
There exist $u_1$, $u_2$
such that $u_1\leq\ul u\leq\hat u\leq u_2$ in some cylinder
$\mc{C}:=B_\delta(x)\times(T-\delta,T]$ and,
at $(x,T)$, $u_1=\ul u=\hat u=u_2$ and
$$\partial_t u_1-D \partial_{xx} u_1+(\mu+l)u_1\leq\nu\rest{\ul v}{y=0}\leq
\nu\rest{\hat v}{y=0}<\partial_t u_2-D \partial_{xx} u_2+(\mu+l)u_2.$$
Since $(x,T)$ is a maximum point for $u_1-u_2$ in $\mc{C}$,
we have that, there, $\partial_t
u_1=\partial_t u_2$ and $\partial_{xx} u_1\leq\partial_{xx} u_2$.
We then get a contradiction with the above strict inequality.
Thus, $\min_{\R}(\hat u-\ul u)(\.,T)>0$ and there exists $(x,y)\in\ol\O$ such
that $(\hat v-\ul v)(x,y,T)=0$.
Using the other two equations of \eqref{Emonotone}, we find $v_1$, $v_2$ such
that $v_1\leq\ul v\leq\hat v\leq v_2$ in a
cylinder $\mc{C}:=B_\delta(x,y)\times(T-\delta,T]$ and, at $(x,y,T)$, $v_1=\ul
v=\hat v=v_2$ and
$$\partial_t v_1-d \Delta v_1\leq h(T,v_1)=h(T,v_2)<
\partial_t v_2-d \Delta v_2\quad \text{if }y>0,$$
$$-d\partial_y v_1+\nu v_1\leq\mu\ul u <
\mu\hat u\leq-d\partial_y v_2+\nu v_2\quad \text{if }y=0.$$
As before, we get a contradiction with the fact that $v_1-v_2$ has a maximum
in $\mc{C}$ at $(x,y,T)$.
\end{proof}

\end{appendix}

\section*{Acknowledgement}
The research leading to these results has received funding from the European
Research Council under the European Union's Seventh Framework Programme
(FP/2007-2013) / ERC Grant Agreement n.321186 - ReaDi -Reaction-Diffusion
Equations, Propagation and Modelling. Part of this work was done while Henri
Berestycki was visiting the University of Chicago. He was also supported by an
NSF FRG grant DMS - 1065979. Luca Rossi was partially supported by the
Fondazione CaRiPaRo Project ``Nonlinear Partial Differential Equations:
models, analysis, and control-theoretic problems''.

\ifx \cedla \undefined \let \cedla = \c \fi\ifx \cftil \undefined \def
  \cftil#1{\~#1} \fi\ifx \cprime \undefined \def \cprime {$\mathsurround=0pt
  '$}\fi\ifx \Dbar \undefined \def \Dbar
  {\leavevmode\raise0.2ex\hbox{--}\kern-0.5emD} \fi\ifx \k \undefined \let \k =
  \c \fi\ifx \lfhook \undefined \let \lfhook = \relax \fi\ifx \mathbb
  \undefined \def \mathbb #1{{\bf #1}}\fi\ifx \polhk \undefined
  \def\polhk#1{\setbox0=\hbox{#1}{\ooalign{\hidewidth
  \lower1.5ex\hbox{`}\hidewidth\crcr\unhbox0}}} \ifx \soft \undefined \def
  \soft{'} \fi

%%%%%%%%%%%%%%%%%%%%%%%%%%%%%%%%%%%%%%%%%%%%%%%%%%%%%%%%%%%%%%%%%%

\end{document}

%The main results of \cite{BRR2} can be summarised as follows: 
%\begin{enumerate}
%\item Besides $(0,0)$, there is a unique stationary state of (\ref{Cauchy}) which is the constant state $(\nu/\mu, 1)$.
 %\item The population $(u,v)$ invades the whole environment, that is, for $(u_0, v_0)\not\equiv (0,0)$ bounded, $\lim_{t\to\infty} (u(t,x), v(t,x)) = (\nu/\mu, 1)$.
 %\item The invasion takes place with an {\em asymptotic speed of spreading}
%$c_*$ in the direction of the line. 
 %\item If $D>2d$ then $c_*>c_K$, or else $c_*=c_K$.
%\end{enumerate}